\documentclass[11pt,A4paper]{article}
\usepackage{newtxmath}
\usepackage{newtxtext,amsthm,mathrsfs,amsfonts,mathtools,braket,feynmf,eucal,secdot,graphicx,labelfig}
\usepackage{tipa,overpic,microtype,float,booktabs,xcolor,cancel,slashed,todonotes,tikz}
\usepackage{subcaption} % Required for \subcaptionbox
\usetikzlibrary{tqft}
\usetikzlibrary{arrows.meta, positioning}
\usepackage[letterpaper]{geometry}
\usepackage[english]{babel}
\usepackage{longtable}
\usepackage[autostyle]{csquotes}
\usepackage{tabularx,colortbl}
\usepackage{caption}
\usepackage[hang,flushmargin]{footmisc}
\usepackage[pdftitle={Ortho-integral surfaces},pdfauthor={Nhat-Minh Doan},ocgcolorlinks,linkcolor=blue,citecolor=orange,urlcolor=red]{hyperref}
\usepackage[small]{titlesec}
\titlespacing{\section}{0pt}{12pt}{0pt}
\titlespacing{\subsection}{0pt}{6pt}{0pt}
\long\def\symbolfootnote[#1]#2{\begingroup
\def\thefootnote{\fnsymbol{footnote}}\footnote[#1]{#2}\endgroup}

\makeatletter
\renewcommand{\@dotsep}{10000} 
\makeatother

\DeclareMathOperator{\arccosh}{arccosh}
\DeclareMathOperator{\arcsinh}{arcsinh}
\newcommand{\bi}{\begin{itemize}}
\newcommand{\ei}{\end{itemize}}
\newcommand{\bq}{\begin{que}}
\newcommand{\eq}{\end{que}}
\newcommand{\be}{\begin{equation*}}
\newcommand{\ee}{\end{equation*}}
\newcommand{\bb}{\mathbb}

\newtheorem{que}{Question}
\newtheorem{thm}{Theorem}[section]
\newtheorem{lem}{Lemma}[section]

\newtheoremstyle{plainnobf}% name
  {}% Space above
  {}% Space below
  {}% Body font, here it is empty indicating no italic
  {}% Indent amount
  {\bfseries}% Theorem head font
  {.}% Punctuation after theorem head
  {.5em}% Space after theorem head
  {}% Theorem head spec (can be left empty, meaning ‘normal’)
\theoremstyle{plainnobf}
\newtheorem{example}{Example}[section]

\newtheorem{dfn}{Definition}[section]

\newtheorem*{rk*}{Remark}
\newtheorem{cor}[thm]{Corollary}
\newtheorem{conj}{Conjecture}[section]
\date{\today}
\begin{document}

\begin{center}
{\Large \bfseries Ortho-integral surfaces}\vspace{.2cm} \\
{\large Nhat Minh Doan\symbolfootnote[1]{Research supported by FNR PRIDE15/10949314/GSM, NRF E-146-00-0029-01, and NAFOSTED 101.04-2023.33. \vspace{.1cm} }}
\end{center}

\allowdisplaybreaks
\begin{abstract}

This paper introduces a combinatorial structure of orthogeodesics on hyperbolic surfaces and presents several relations among them. As a primary application, we propose a recursive method for computing the trace (the hyperbolic cosine of the length) of orthogeodesics and establish the existence of surfaces where the trace of each orthogeodesic is an integer. These surfaces and their orthogeodesics are closely related to integral Apollonian circle packings. Notably, we found a new type of root-flipping that transitions between roots in different quadratic Diophantine equations of a certain type, with Vieta root-flipping as a special case. Finally, we provide a combinatorial proof of Basmajian's identity for hyperbolic surfaces, akin to Bowditch's combinatorial proof of the McShane identity.

\end{abstract}

% {\setlength{\parskip}{0pt}\hypersetup{linkcolor=black}\tableofcontents}

\section{Introduction} 

The set of orthogeodesics, first introduced in \cite{basmajian1993orthogonal} by Basmajian in the early 90s, consists of geodesic arcs perpendicular to the boundary of $M^n$ at their endpoints, where $M^n$ is a compact hyperbolic $n$-manifold with non-empty totally geodesic boundary.  In that work, Basmajian proved a beautiful identity that connects the ortho-length spectrum with the total volume of the boundary of the hyperbolic manifold. Over the past three decades, significant progress has been made in understanding the set of orthogeodesics across various contexts and its applications,  see e.g. \cite{basmajian2024orthosystoles,basmajian2020prime,bell2023counting,belolipetsky2022lower,bridgeman2011orthospectra,bridgeman2021dilogarithm,bridgeman2010hyperbolic, bridgeman2014moments,calegari2010chimneys,fanoni2020basmajian, he2018basmajian,   he2018prime, jaipong2023dilogarithm,   parlier2020geodesic,vlamis2015moments} and the references
therein. Notably, Bridgeman-Kahn introduced another fundamental identity, detailed in \cite{bridgeman2011orthospectra,bridgeman2010hyperbolic}, linking the ortho-length spectrum on $M^n$ to the volume of the unit tangent bundle of $M^n$. In dimension two, the identities formulated by Basmajian and Bridgeman can be expressed respectively as follows:
\begin{equation}\label{equ:Bas&Bri}
    \ell(\partial{M^2})=\sum_{\eta}2\log(\coth(\ell(\eta)/2)), \,\,\,\,\,\,\,\,\,\,\,\,\,\,\,-\frac{\pi^2}{2} \chi(M^2)=\sum_{\eta}\mathcal{L}\left(1/{\cosh^2(\ell(\eta)/2)}\right)
\end{equation} 
where both sums run over the set of orthogeodesics on the surface $M^2$, and $\mathcal{L}$ is the Rogers dilogarithm function, see \cite{kirillov1995dilogarithm,zagier2007dilogarithm}. Specifically, Bridgeman's identity on certain special surfaces leads to the classical dilogarithm identities and infinitely many new ones, as shown in \cite{bridgeman2021dilogarithm,jaipong2023dilogarithm}. Zagier noted in \cite{zagier2007dilogarithm} that the dilogarithm appears not only in hyperbolic geometry but also in algebraic K-theory and mathematical physics, particularly in conformal field theory. The relations involving the dilogarithm continue to attract considerable interest in the K-theory and cluster algebra communities.

Influenced by these results, our initial goal in studying the set of orthogeodesics was to precisely describe the dilogarithm identities derived from Bridgeman's identity on pairs of pants. This journey leads to a combinatorial structure on the set of oriented orthogeodesics and identity relations, which leads us to a combinatorial proof of Basmajian's identity by using the approach in Bowditch's paper, see \cite{bowditch1996proof,  bowditch1998markoff}. Bowditch's method was originally used to give a combinatorial proof of McShane's identity, see \cite{mcshane1998simple}, then, later on, was applied to different contexts to obtain many other descendants and generalizations, see \cite{hu2014new,tan2006necessary,tan2008generalized}. %(\cite{B98},). 
Furthermore, In \cite{labourie2018probabilistic}, Labourie and Tan extended Bowditch's idea to provide a planar tree coding of oriented simple orthogeodesics on hyperbolic surfaces, along with a probabilistic explanation of McShane's identity for higher genus surfaces. They created this coding by performing special flips on the space of ideal triangulations of the surface. When an ideal triangulation is lifted to the universal cover of the surface, their method performs flips for all lifts of a simple orthogeodesic at each step.

To achieve coding of oriented orthogeodesics, we modify their algorithm to flip only one lift of a simple orthogeodesic at each step. Additionally, in Section \ref{coherentsubset}, we extend this method to construct a coding applicable to any suitable subset of oriented orthogeodesics starting from a boundary component. This construction could be useful for studying coherent markings introduced by Parlier in \cite{parlier2020geodesic} and prime orthogeodesics introduced by Basmajian, Parlier, and Tan in \cite{basmajian2020prime}.

In distance geometry, the Cayley-Menger determinant allows for the computation of the volume of an \( n \)-simplex in Euclidean space using the squares of the distances between all pairs of its vertices (see \cite{cayley1841theorem}). By applying hyperbolic trigonometric formulas for pentagons and hexagons, we derive several identity relations among orthogeodesics and connect them to a type of Cayley-Menger determinant and Penner's Ptolemy relation. 

It is important to note that Penner's Ptolemy relation involves \textit{$\lambda$-lengths} of orthogeodesics on surfaces with cusps, which has been shown to be significant in various contexts, such as Penner's theory of the moduli space of surfaces with punctures (see \cite{penner1987decorated, penner2012decorated}) and, more recently, cluster algebras (see, e.g., \cite{fomin2018cluster} and related works). Surprisingly, by extending Penner's Ptolemy relation to surfaces with non-empty totally geodesic boundaries, we find that the \textit{traces} (the hyperbolic cosines of the lengths) of three neighboring orthogeodesics \((X, Y, Z)\) with respect to given initial data \((a, b, c)\) satisfy a quadratic Diophantine equation as follows:
\[
\resizebox{\textwidth}{!}{
\fbox{$(a^2-1)X^2 + (b^2-1)Y^2 + (c^2-1)Z^2 - 2(ab+c)XY - 2(bc+a)YZ - 2(ca+b)XZ = a^2 + b^2 + c^2 + 2abc - 1$}
}
\]
For hyperbolic surfaces with at least one boundary component, we introduce a generalization of Penner's Ptolemy relation, termed the \textit{mixed Ptolemy relation}. This relation provides us with similar but distinct Diophantine equations (see Lemma \ref{mix}). Utilizing these equations and the combinatorial structure of orthogeodesics, we propose a recursive method for computing the trace of orthogeodesics on the surface.

In particular, we found a new type of root flipping, termed \textit{inter-root flipping}. This type of root flipping transitions between roots in different quadratic Diophantine equations, distinguishing it from Vieta root flipping (a.k.a. Vieta jumping). By using these Diophantine equations and the inter-root flipping, we establish the existence of \textit{ortho-integral surfaces}, where each orthogeodesic trace is an integer. Examples of these surfaces and their corresponding quadratic Diophantine equations are provided in Table \ref{table:2}. This intriguing phenomenon shares many similarities with integral Apollonian circle packings, which have attracted significant interest due to their connections with thin groups and the local-to-global principle for quadratic forms (see e.g. \cite{bourgain2014local, fuchs2019local, haag2023local} and the references therein).

The construction of the Apollonian gasket begins with three circles \( C_1 \), \( C_2 \), and \( C_3 \), each tangent to the other two but without a single point of triple tangency. These circles can vary in size, with the possibility of two being inside the third, or all three being outside each other. Apollonius discovered two additional circles, \( C_4 \) and \( C_5 \), tangent to all three original circles - these are known as Apollonian circles. For four mutually tangent circles with curvatures \( k_i \) for \( i = 1, \ldots, 4 \), Descartes' theorem states:
\[
\fbox{$(k_1 + k_2 + k_3 + k_4)^2 = 2 (k_1^2 + k_2^2 + k_3^2 + k_4^2)$}
\]

An Apollonian circle packing exhibits the property where the sum of the areas of all circles within the largest circle equals the area of the largest circle itself. Similarly, the set of orthogeodesics satisfies Basmajian's and Bridgeman-Kahn's identities as mentioned earlier. Moreover, each term in Basmajian's identity is twice the hyperbolic area of an \((n-1)\)-ball on \(\partial M^n\) corresponding to each orthogeodesic. These \((n-1)\)-balls are referred to as \textit{leopard spots} by Thurston and \textit{chimney bases} by Calegari (see \cite{calegari2010chimneys}).

Table \ref{table:1} outlines some common features between the Apollonian circle packing and the set of orthogeodesics. It is worth mentioning that there is a similar correspondence with the set of simple geodesics on a once-punctured torus: $\lambda$-length of simple orthogeodesics (or the trace of simple closed geodesics), McShane's term, Markov's equation, and Vieta jumping. All of these problems fall within a common framework that involves a set with a combinatorial structure, Diophantine equations, and a type of root flipping. More generally, this framework involves a family of algebraic varieties and a method for jumping between integer points on them. Thus, it is both natural and intriguing to explore various number-theoretic questions related to the set of orthogeodesics on ortho-integral hyperbolic manifolds in two and higher dimensions.

\begin{table}[ht]
\centering
\begin{tabularx}{\textwidth}{|p{7cm}|X|}
  \hline

\cellcolor{gray!30} \textbf{Apollonian gasket} & \cellcolor{gray!30} \textbf{The set of orthogeodesics}  \\
  \hline
  The curvature of circles & Trace of orthogeodesics\\
  \hline
  Area of circles & Basmajian's or Bridgeman-Kahn's term \\
  \hline
  Descartes's equation & A collection of quadratic Diophantine equations \\
  \hline
  Vieta root flipping  & Inter-root flipping \\
  \hline
\end{tabularx}
\caption{Apollonian circle packing versus the set of orthogeodesics.}
\label{table:1}
\end{table}

\textbf{Main results.} In this paper, we will focus solely on dimension two. Let $S$ be an orientable hyperbolic surface with non-empty totally geodesic boundary. Recall that $S$ is \textit{ortho-integral} if the trace (the hyperbolic cosine of the length) of each orthogeodesic on the surface is an integer. An \textit{orthobasis}, denoted by $\mathcal{B}$, on $S$ is a set of pairwise disjoint simple orthogeodesics that decompose the surface into hexagons. As a slight abuse of notation, we identify an orthobasis with the set of traces of orthogeodesics in the orthobasis. Our first result below provides a list of ortho-integral pairs of pants and one-holed tori.
\begin{thm}\label{thm:orthointegralsurface} (i) A pair of pants is ortho-integral if it admits an orthobasis $(a,b,c)$ belonging to the following list:  
$$(2,2,2),(3,3,3),(2,2,5),(3,3,7),(5,5,11).$$
(ii) A one-holed torus is ortho-integral if it admits an orthobasis $(a,b,c)$ belonging to the following list:   
$$(2, 2, 2),(2, 2, 3),(2, 2, 5),(2, 3, 6),(2, 4, 4),(2, 4, 7),(2, 5, 8),(2, 7, 10),(2, 13, 16),$$
$$(3, 3, 3),(3, 3, 7),(3, 5, 5),(3, 5, 9),(3, 9, 13),(3, 17, 21),(4, 4, 5),(4, 5, 10),(4, 6, 6),(4, 11, 16),$$
$$(5, 5, 11),(5, 7, 13),(5, 13, 19),(6, 6, 9),(6, 36, 64),(7, 9, 9),(7, 17, 25),(8, 13, 22),(9, 11, 21),$$
$$(10, 10, 17),(10, 12, 12),(11, 49, 61),(13, 29, 43),(17, 19, 37),(19, 21, 21).$$
\end{thm}

Note that each pair of pants has at most 4 distinct orthobases. We refer to Figure \ref{fig:integralothopants} for the set of pairs of pants satisfying condition $(i)$ in Theorem \ref{thm:orthointegralsurface}. There are seven such pairs of pants, including additional orthobases $(2,2,17)$ for the pair of pants with orthobasis $(2,2,2)$, and $(3,3,19)$ for the pair of pants with orthobasis $(3,3,3)$. In contrast, each one-holed torus has infinitely many distinct orthobases.

\begin{figure}[h]

\centering
\begin{tikzpicture}[scale=0.45] % Adjust scale if necessary
  % First row of four instances
  \foreach \x/\labelA/\labelB/\labelC in {0/2/2/2, 1/2/2/17, 2/3/3/3, 3/3/3/19, 4/2/2/5, 5/3/3/7, 6/5/5/11} {
    \begin{scope}[shift={(\x*5,0)}] % Increase horizontal shift (5) for more space
      % Ellipses
      \draw (0,0) ellipse (.5 and .1);
      \draw (-1,-2) ellipse (.5 and .1);
      \draw (1,-2) ellipse (.5 and .1);
      
      % Curves
      \draw (-.5,0) to[out=-90,in=90] node[midway, left] {\labelA} (-1.5,-2);
      \draw (.5,0) to[out=-90,in=90] node[midway, right] {\labelB} (1.5,-2);
      \draw (-.5,-2) to[out=90,in=90] node[midway, below] {\labelC} (.5,-2);
    \end{scope}
  }
\end{tikzpicture}
\captionsetup{justification=centering}
\caption{The set of all ortho-integral pairs of pants that admit an orthobasis $(a,b,c)\in\{(2,2,2),(3,3,3),(2,2,5),(3,3,7),(5,5,11)\}.$}
\label{fig:integralothopants}
\end{figure}

Two oriented orthogeodesics on \( S \) are considered (Farey) \textit{neighbors} (with respect to an orthobasis \(\mathcal{B}\)) if they start from the same boundary component of \( S \) and, together with an element of \(\mathcal{B}\) and some segments on \(\partial S\), form an immersed hexagon on \( S \) (see Definition \ref{dfn:neighboringorthogeodesics} for more details).
The main tool in proving Theorem \ref{thm:orthointegralsurface} is the quadratic Diophantine equation
\begin{equation}\label{equ:Ptolemygeodesics}
    (a^2-1)X^2 + (b^2-1)Y^2 + (c^2-1)Z^2 - 2(ab+c)XY - 2(bc+a)YZ - 2(ca+b)XZ = a^2 + b^2 + c^2 + 2abc - 1
\end{equation}
 for three neighboring orthogeodesics with traces $(X,Y,Z)$ and the so-called inter-root flipping. When the surface is a pair of pants with an orthobasis \((a, b, c)\), the inter-root flipping corresponds to Vieta jumping between roots of the quadratic Diophantine equation \ref{equ:Ptolemygeodesics}. However, it is important to emphasize that when the surface differs from a pair of pants, inter-root flipping does not equate to Vieta jumping. An example illustrating inter-root flipping for one-holed tori is provided in the next result (Theorem \ref{thm:panttorusjumpintro}).

  The following theorem can be generalized to any hyperbolic surface with non-empty boundary. However, for simplicity, we will only state it in this paper for pairs of pants and one-holed tori.

\begin{thm}\label{thm:panttorusjumpintro} We denote a pair of pants with an orthobasis \((a,b,c)\) by \( P(a,b,c) \) and a one-holed torus with an orthobasis \((a,b,c)\) by \( T(a,b,c) \), then
\bi
\item[(i)] If $\textbf{v}=(X,Y,Z)^T$ is a vector of traces of three pairwise neighboring orthogeodesics on $P(a,b,c)$, then all of the traces of orthogeodesics starting from the same boundary component with $X,Y,Z$ on $P(a,b,c)$ are given by the coordinates of vectors in the orbit $G_{P(a,b,c)}\cdot \textbf{v}$, where $G_{P(a,b,c)}$ is a group
generated by
$$f_X=\begin{bmatrix}
    -1 & \frac{2(ab+c)}{a^2-1}& \frac{2(ac+b)}{a^2-1}\\
    0 & 1& 0\\
    0&0&1
\end{bmatrix},f_Y=\begin{bmatrix}
    1 & 0& 0\\
    \frac{2(ba+c)}{b^2-1} & -1& \frac{2(bc+a)}{b^2-1}\\
    0&0&1
\end{bmatrix},
f_Z=\begin{bmatrix}
    1 & 0 & 0\\
    0 & 1& 0\\
    \frac{2(ca+b)}{c^2-1}&\frac{2(cb+a)}{c^2-1}& -1
\end{bmatrix}.$$
\item[(ii)] If $\textbf{v}=(X,Y,Z)^T$ is a vector of traces of three pairwise neighboring orthogeodesics on $T(a,b,c)$, then all of the traces of orthogeodesics on $T(a,b,c)$ are given by the coordinates of vectors in the orbit $G_{T(a,b,c)}\cdot \textbf{v}$, where $G_{T(a,b,c)}$ is a group generated by
$$g_X=\begin{bmatrix}
    -1 & \frac{b+c}{a-1}& \frac{b+c}{a-1}\\
    0 & 0& 1\\
    0&1&0
\end{bmatrix},g_Y=\begin{bmatrix}
    0 & 0& 1\\
    \frac{a+c}{b-1} & -1& \frac{a+c}{b-1}\\
    1&0&0
\end{bmatrix},
g_Z=\begin{bmatrix}
    0 & 1 & 0\\
    1 & 0& 0\\
    \frac{a+b}{c-1}&\frac{a+b}{c-1}& -1
\end{bmatrix}.$$
\ei
\end{thm}

In general, inter-root flipping is a transition between roots in various quadratic Diophantine equations associated with a given orthobasis on a hyperbolic surface with a non-empty boundary. For instance, on a generic four-holed sphere, one can select an orthobasis that corresponds to at least two distinct Diophantine equations, with inter-root flipping transitioning between the roots of these equations.

The Diophantine equation \ref{equ:Ptolemygeodesics} and the inter-root flipping are among several other identity relations concerning the distances between geodesics and horocycles on the hyperbolic plane that we will introduce in Section \ref{appendix}. %These include harmonic relations, Ptolemy relations of geodesics, mixed Ptolemy relations, relations of quintets of geodesics/horocycles, and their special cases known as \textit{orthoshapes} (ortho-isosceles trapezoid, orthorectangle, orthokite, orthoparallelogram relation). 
% Some of these relations are restricted to the tree of orthogeodesics on hyperbolic surfaces (see Section \ref{section 1}) and are expressed as recursive formulas, isosceles trapezoid, rectangle, kite, parallelogram, edge and vertex relations. 
These relations are connected to a type of Cayley-Menger determinant and will be presented in the next theorem.

Let \( U \) and \( V \) be two arbitrarily disjoint geodesics or horocycles in \(\mathbb{H}\). Denote \(\kappa_U\) and \(\kappa_V\) as their respect geodesic curvatures. For convenience, we define $\overline{UV}$ as: $$\overline{UV} :=\frac{e^{\frac{1}{2}d_{\mathbb{H}}(U,V)}+(1-\kappa_U)(1-\kappa_V)e^{-\frac{1}{2}d_{\mathbb{H}}(U,V)}}{2},$$ where $d_{\mathbb{H}}(U,V)$ is the hyperbolic distance between $U$ and $V$. We obtain the following relation in the form of the Cayley-Menger determinant:
\begin{thm} \label{Thm: Determinant}Let $\{A_1, A_2,A_3, A_4\}$ be the set of four disjoint geodesics/horocycles in $\mathbb{H}$, each of them divides $\mathbb{H}$ into two domains such that the other three lie in the same domain. Then

\begin{equation}\label{equ:Cayley-Menger}
\det\begin{bmatrix}
2& 1-\kappa_{A_1} & 1-\kappa_{A_2}&1-\kappa_{A_3}&1-\kappa_{A_{4}}\\
1-\kappa_{A_1} & 0 & \overline{A_1A_2}^2 & \overline{A_1A_3}^2 &\overline{A_1A_4}^2\\
1-\kappa_{A_2}&\overline{A_2A_1}^2&0&\overline{A_2A_3}^2 &\overline{A_2A_4}^2\\
1-\kappa_{A_3}&\overline{A_3A_1}^2&\overline{A_3A_2}^2&0&\overline{A_3A_{4}}^2\\
1-\kappa_{A_{4}}&\overline{A_{4}A_1}^2&\overline{A_{4}A_2}^2 & \overline{A_{4}A_{3}}^2&0
\end{bmatrix}=0.    
\end{equation}
\end{thm}
Note that in the special case where \(\kappa_{A_i} = 1\) for all \(i \in \{1, 2, 3, 4\}\), Equation \ref{equ:Cayley-Menger} corresponds to Penner's Ptolemy relation \cite{penner1987decorated}. When \(\kappa_{A_i} = 0\) for all \(i \in \{1, 2, 3, 4\}\), Equation \ref{equ:Cayley-Menger} simplifies to the Ptolemy relation for geodesics (see Equation \ref{equ:Ptolemygeodesics}). We believe that this theorem could potentially be extended to hyperbolic spaces of higher dimensions.

Our next result concerns a combinatorial structure of the set of orthogeodesics, which underlies all the aforementioned results. Furthermore, this structure also leads to a discretization of Basmajian's identity. 

Let $S$ be an orientable hyperbolic surface with boundary $\partial S$ consisting of simple closed geodesics. Let $\eta_{\frac{0}{1}}$ and $\eta_{\frac{1}{1}}$ be two oriented orthogeodesics starting from a simple closed geodesic $\alpha$ at the boundary of $S$. The starting points of these two orthogeodesics divide $\alpha$ into two open subsegments, namely $\alpha_1$ and $\alpha_2$. Suppose that $\alpha_1 \neq \varnothing$. Denote by $\mathcal{O}_{\alpha_1}$ the set of oriented orthogeodesics starting from $\alpha_1$. 

Let $T_1$ be a planar rooted trivalent tree whose first vertex is of valence 1, and all other vertices are of valence 3. Let $E(T_1)$ be the set of edges of $T_1$. Each edge of the tree has two sides associated with two neighboring complementary regions of the tree (see Figure \ref{panttree} for an illustration). Let $\Omega(T_1)$ be the set of complementary regions of the tree. 

In Section \ref{section 1}, we show that each orthobasis on $S$ induces an order-preserving bijection between $\mathcal{O}_{\alpha_1} \sqcup \{\eta_{\frac{0}{1}}\}\sqcup \{\eta_{\frac{1}{1}}\}$ and $\Omega(T_1)$. This bijection allows us to state Basmajian's identity in a purely combinatorial manner and provides a proof of the identity independent of Basmajian's original geometric proof.

 We define a weight map $\Phi$ on the set of edges and complementary regions of the tree $T_1$. This map satisfies conditions coming from the fact that we want them to correspond to the cosh length function $\cosh(\ell(.))$. In particular, the map $\Phi: E(T_1) \sqcup \Omega(T_1) \to (1,\infty)$ has a harmonic relation at any vertex except at the root of the tree (see Equation \ref{harmonic}). Basmajian's identity for the tree $T_1$ can be expressed in the following form:
\begin{thm}\label{Basmaj}(Basmajian's identity for $T_1$)  If  $\sup\{\Phi(x) | x \in E(T_1) \}<\infty$, then
$$\log\left(\frac{x_0+Y_0Z_0+\sqrt{x_0^2+Y_0^2+Z_0^2+2x_0Y_0Z_0-1}}{(Y_0-1)(Z_0-1)}\right)=\sum_{X\in \Omega(T_1)}\log\left(\frac{X+1}{X-1}\right)$$
where $(x_0,Y_0,Z_0)$ is the initial edge region triple at the root of the tree. Note that $x_0,X,Y_0,Z_0$ are the abbreviations of $\Phi(x_0),\Phi(X),\Phi(Y_0),\Phi(Z_0)$ respectively.
\end{thm}
% The following corollary is a combinatorial form of Basmajian's identity for the set of oriented orthogeodesics starting from a simple closed geodesic, say $o_1$, on the boundary of a hyperbolic surface. Suppose that $o_1$ is divided into $n$ subsegments by a hexagonal decomposition.
% \begin{cor}(Basmajian's identity for $T_n$) 
% Let $T_n$ be a rooted trivalent tree with $n$ edges starting from the root. If  $\sup\{\Phi(x) | x \in E(T_n) \}<\infty$, then 
% $$\sum_{k=1}^n \arccosh\left(\frac{x_{0,k}+Y_{0,k}Z_{0,k}}{\sqrt{(Y_{0,k}^2-1)(Z_{0,k}^2-1)}}\right)=\label{Basmajian}\sum_{X\in \Omega(T_n)}\log\left(\frac{X+1}{X-1}\right),$$
% where $(x_{0,k},Y_{0,k},Z_{0,k})$'s are edge region triples at the root of the tree. Note that $Z_{0,k}=Y_{0,k+1}$, with $k\in \{1,2,...,n\}$, and $Y_{0,n+1}:=Y_{0,1}$.
% \end{cor}
We emphasize that the index set of the above identity is derived from the complementary regions of a rooted trivalent tree. The condition $\sup\{\Phi(x) \mid x \in E(T_1)\} < \infty$ always holds when $S$ is a surface of finite type.

Lastly, we investigate certain types of e-orthoshapes (see Definition \ref{r+}). Generally, an e-orthoshape is a set of finite orthogeodesics that satisfy specific conditions for any hyperbolic structure on a surface. In this paper, we focus on e-orthoshapes related to length-equivalent orthogeodesics. Let \( S \) be an arbitrary hyperbolic surface with totally geodesic boundaries. We show that:

\begin{thm}
The involution reflections on an immersed pair of pants yield infinitely many e-ortho-isosceles-trapezoids, e-orthorectangles, and e-orthokites on \( S \). However, there are no e-orthosquares on \( S \).
\end{thm}
The existence of these types of e-orthoshapes, especially e-orthorectangles, on surfaces simplifies significantly the inter-root flipping for some subsets of triples of orthogeodesics. In particular, it explains why there exist quadruples of orthogeodesics $(X, Y, Z, T)$ on surfaces that always satisfy the following Pythagorean-type relation $$\cosh^2\left(\frac{\ell(X)}{2}\right) + \cosh^2\left(\frac{\ell(T)}{2}\right) = \cosh^2\left(\frac{\ell(Y)}{2}\right) + \cosh^2\left(\frac{\ell(Z)}{2}\right)$$ no matter what hyperbolic structure is endowed on the surface. One can ask whether all of these e-orthoshapes arise from the reflection involutions on immersed pairs of pants. This question is intimately related to the length-equivalent problem studied in \cite{anderson2003variations} and \cite{leininger2003equivalent}.

\textbf{Structure.} Sections \ref{Pre} and \ref{Notations} contain necessary notations used throughout this paper. Section \ref{section 1} describes a combinatorial structure on a subset of oriented orthogeodesics, related identity relations.  Section \ref{applications} will be about some applications: Ortho-integral surfaces, some infinitely (dilogarithm) identities, and a combinatorial proof of Basmajian's identity. Section \ref{reflection} provides a construction of some types of e-orthoshapes using reflection involutions. Section \ref{appendix} introduces several identity relations of geodesics and horocycles on the hyperbolic plane. Section \ref{questionremarks} will be about some remarks and further questions.

\textbf{Acknowledgments.} I would like to express my deep gratitude to my advisor, Hugo Parlier, for his invaluable discussions and constant encouragement. I am especially thankful to Ser Peow Tan for carefully reading the initial manuscripts, providing detailed comments, and offering useful suggestions that greatly improved this paper. I also thank Greg McShane for pointing out an unclear argument in the proof of Theorem \ref{Thm: integral surfaces} and for his useful suggestions. Lastly, I extend my thanks to Binbin Xu for generously dedicating his time to answer many of my questions and to David Fisac Camara for engaging in helpful discussions.

\section{Orthobasis, (e-)orthoshapes, and length equivalent orthogeodesics}\label{Pre}
%The hyperbolic plane %is a complete simply connected two dimensional Riemannian manifold of constant Gaussian curvature $-1$. It 
%can be modeled by the upper half plane with a Riemannian metric. 
%$$\left(\bb{H}=\{z\in \bb{C}|\Im(z)>0\},ds^2=\frac{dx^2+dy^2}{y^2}\right).$$ 
%Under the Cayley transform $z \to \frac{z-i}{z+i}$, one obtains the Poincaré disk model. %which is usually used for illustrations. %There are several other useful models, one of them is the hyperboloid model, as mentioned by Penner \cite{Pennerbook}, it is useful both for calculations and for figures. However, we are familiar with the above two models and will use them through out this paper. 

Let $\Sigma$ be a topological surface of negative Euler characteristic with $n$ punctures, $n\geq 1$. Denote by $\partial{\Sigma}$ the set of punctures on $\Sigma$. The Teichm\"{u}ller space $\mathcal{T}(\Sigma)$ is denoted as the set of marked hyperbolic structures on $\Sigma$ up to isotopy such that there are $b$ geodesic boundaries and $c$ cusps where $b+c=n$.

Let $\eta$ be a homotopy class of arcs relative to $\partial{\Sigma}$. The geometric realization of $\eta$ with respect to an $S\in \mathcal{T}(\Sigma)$ is an \textbf{orthogeodesic} which is a geodesic arc perpendicular to $\partial{S}$ at both ends, with a convention that any geodesic with an endpoint at a cusp is said to be perpendicular to that cusp. Note that an orthogeodesic is of infinite length if one of its endpoints is at a cusp of $S$.

Now we introduce the notions of orthotriangle, orthobasis, standard orthobasis, neighboring orthogeodesics, and orthotriangulation.
\begin{dfn}\label{orthotriangle} An \textbf{orthotriangle} on $S \in \mathcal{T}(\Sigma)$ is an immersed polygon with its boundary consisting of three orthogeodesics and at most three geodesic subsegments of $\partial{S}$ (i.e. it is an $n$-gon with $(6-n)$ ideal vertices, where $n\in\{3,4,5,6\}$). \end{dfn}
Note that with this definition, orthotriangles are related to the so-called generalized triangles in Buser's book \cite{buser2010geometry}.

\begin{dfn}\label{orthobasis} An \textbf{orthobasis} on $S \in \mathcal{T}(\Sigma)$ is a collection of pairwise disjoint simple orthogeodesics which decomposes $S$ into a collection of embedded orthotriangles. The set of all the orthotriangles is called an \textbf{orthotriangulation}.
\end{dfn}
\begin{dfn}\label{standard}
An orthotriangulation $\Delta$ is \textbf{standard} if, for any orthotriangle in $\Delta$, the three orthogeodesics at its boundary are pairwise distinct.
\end{dfn}

\begin{dfn}\label{dfn:neighboringorthogeodesics}
Let \( S \in \mathcal{T}(\Sigma)\) with an orthobasis \(\mathcal{B}\). Two oriented orthogeodesics $\eta_1$ and $\eta_2$ on \( S \) are \textbf{(Farey) neighbors} (with respect to \(\mathcal{B}\)) if they start from the same boundary component and, together with an element $\eta \in \mathcal{B}$, form an orthotriangle on \( S \).
\end{dfn}

Now we introduce the notions of truncated length of orthogeodesics, orthoquadrilateral, ortho-isosceles-trapezoid, orthorectangle, orthokite, orthoparallelogram on $\Sigma$. 

Let $S \in \mathcal{T}(\Sigma)$. A \textbf{truncated surface} on $S$, say $\hat{S}$, is a hyperbolic surface with its boundary $\partial{\hat{S}}$ consisting of simple closed geodesics and/or simple closed horocycles obtained from cutting off all the cusp regions of $S$. Denote by $\hat{S}(2)$ the \textbf{natural truncated surface} of $S$ where all removing cusp regions are of the same area $2$. Note that $\hat{S}=\hat{S}(2)=S$ if there are no cusps on $S$. We define $\hat{\mathcal{T}}(\Sigma)$ to be the set of all pairs $(S,\hat{S})$ with $S \in \mathcal{T}(\Sigma)$. Similarly, a \textbf{natural concave core} (or a concave core of grade 1 as defined in \cite{basmajian2020prime}), denoted by $S^\star{}$, is the surface obtained by cutting off all natural collars of cusps and simple closed geodesics on $\partial{S}$ of $S$. Note that a \textbf{natural collar} of a boundary component $\beta \in \partial{S}$ is
\vspace{-3mm}
\begin{itemize}
\item a cusp region of area $2$ surrounding $\beta$ if $\beta$ is a cusp,
\item a set of points at distance less than $\arcsinh\left(1/\sinh(\ell(\beta)/2)\right)$ from $\beta$ if $\beta$ is a simple closed geodesic.
\end{itemize}
\vspace{-3mm}

Let $\bar{S}$ be either $\hat{S}$ or $S^\star{}$ defined as above. Each orthogeodesic $\eta$ on $S$ will be truncated by $\partial{\bar{S}}$ and associated to a truncated orthogeodesic which is a subarc of $\eta$ perpendicular to $\partial \bar{S}$ at both endpoints. Denote by $\ell_{\hat{S}}(\eta)$ and $\ell_{S^\star{}}(\eta)$  the \textbf{truncated lengths} of $\eta$ with respect to $\hat{S}$ and $S^\star{}$ on $S$ and let
$$\ell_{\bar{S}}(\eta):=\left\{ \begin{array}{rcl}
\ell_{\hat{S}}(\eta) & \mbox{if}
& \bar{S}=\hat{S} \\ \ell_{S^\star{}}(\eta)
& \mbox{if} & \bar{S}=S^\star{}
\end{array}\right.
$$

\begin{dfn}\label{e-shapes}
 An \textbf{orthoquadrilateral} \( \mathbb{O} := XYZT \) on \( S \in \mathcal{T}(\Sigma) \) is an immersed polygon with its boundary consisting of four orthogeodesics, namely \( (XY, YZ, ZT, TX) \) in cyclic order, and at most four geodesic segments of \( \partial{S} \) (i.e., it is an \( n \)-gon with \( (8 - n) \) ideal vertices, where \( n \in \{4, 5, 6, 7, 8\} \)). Let \( XZ \) and \( YT \) be diagonal orthogeodesics of \( \mathbb{O} \). Let \( \bar{S} \) be either the natural concave core \( S^\star \) or a truncated surface \( \hat{S} \) on \( S \). Then with respect to \( \bar{S} \),

\vspace{-3mm}
\begin{itemize}
\item{}$\mathbb{O}$ is an ortho-isosceles-trapezoid if $\ell_{\bar{S}}(XT)=\ell_{\bar{S}}(YZ)$ and $\ell_{\bar{S}}(XZ)=\ell_{\bar{S}}(YT)$,
\item{}$\mathbb{O}$ is an orthorectangle if $\ell_{\bar{S}}(XT)=\ell_{\bar{S}}(YZ)$, $\ell_{\bar{S}}(XY)=\ell_{\bar{S}}(ZT)$ and $\ell_{\bar{S}}(XZ)=\ell_{\bar{S}}(YT)$,
\item{}$\mathbb{O}$ is an orthokite if $\ell_{\bar{S}}(XY)=\ell_{\bar{S}}(XT)$ and $\ell_{\bar{S}}(ZY)=\ell_{\bar{S}}(ZT)$,
\item{}$\mathbb{O}$ is an orthoparallelogram if $\ell_{\bar{S}}(XY)=\ell_{\bar{S}}(ZT)$ and $\ell_{\bar{S}}(XT)=\ell_{\bar{S}}(YZ)$,

\item{}$\mathbb{O}$ is a orthorhombus if $\ell_{\bar{S}}(XT)=\ell_{\bar{S}}(YZ)=\ell_{\bar{S}}(XY)=\ell_{\bar{S}}(ZT)$,

\item{}$\mathbb{O}$ is an orthosquare if $\ell_{\bar{S}}(XT)=\ell_{\bar{S}}(YZ)=\ell_{\bar{S}}(XY)=\ell_{\bar{S}}(ZT)$ and $\ell_{\bar{S}}(XZ)=\ell_{\bar{S}}(YT)$.

\end{itemize}
\end{dfn}

Finally, we introduce the notion of length equivalent orthogeodesics and e-orthoshapes. 

Two homotopy classes of arcs relative to $\partial{\Sigma}$ are length equivalent in $\mathcal{T}(\Sigma)$ if their orthogeodesics are of the same truncated length with respect to any pair $(S,\hat{S}) \in \hat{\mathcal{T}}(\Sigma)$.  One can show that if two orthogeodesics are length equivalent, then their endpoints are on the same pair of elements in $\partial{\Sigma}$. Thus, we will use the following definition for length equivalent orthogeodesics:
\begin{dfn}\label{dfn:lengthequivalentorthogeodesics}
    Two orthogeodesics, say $\eta_1$ and $\eta_2$, are \textbf{length equivalent} if $\ell_{S^\star{}}(\eta_1)=\ell_{S^\star{}}(\eta_2)$ for all $S\in \mathcal{T}(\Sigma)$.
\end{dfn}

There is an equivalent way to define length-equivalent orthogeodesics based on length-equivalent closed geodesics. Indeed, each orthogeodesic $\eta$ is roughly a seam between two boundary components, say $\beta_1$ and $\beta_2$, of a collection of immersed pairs of pants on $S$. Among these immersed pairs of pants, there is a unique maximal one, say $P^{*}$, which contains all the rest. One can associate $\eta$ to the closed geodesic, say $\beta_3$, at the remaining boundary component of $P^{*}$ (see the precise definition in \cite{basmajian2020prime} or \cite{bell2023counting}). Then, two orthogeodesics are \textbf{length equivalent} if their associated closed geodesics are of the same length for all $S \in \mathcal{T}(\Sigma)$. Here is the relation between $\ell_{S^\star{}}(\eta)$ and $\beta_1$, $\beta_2$ and $\beta_3$:
$$\ell_{S^\star{}}(\eta)=\arccosh\frac{a_1a_2+a_3}{\sqrt{(a_1^2-1)(a_2^2-1)}}-\arcsinh\frac{1}{\sqrt{a_1^2-1}}-\arcsinh\frac{1}{\sqrt{a_2^2-1}}$$
$$=\log\left(\frac{a_1a_2+a_3+\sqrt{a_1^2+a_2^2+a_3^2+2a_1a_2a_3-1}}{(a_1+1)(a_2+1)} \right),$$
in which $a_i$ denotes the half-trace of $\beta_i$, i.e. $\cosh(\ell(\beta_i)/2)$, for any $i\in\{1,2,3\}$. Observe that $e^{\ell_{S^\star{}}(\eta)}$ is a root of the following equation: $(a_1+1)(a_2+1)x^2-2(a_1a_2+a_3)x+(a_1-1)(a_2-1)=0.$

%Let $\mathcal{T}_b(\Sigma) \subset \mathcal{T}(\Sigma)$ be the set of marked hyperbolic structures of totally geodesic boundary. In fact, $\mathcal{T}(\Sigma)$ is the closure of $\mathcal{T}_b(\Sigma)$ with respect to ??? metric. We will show that\begin{lem}Two orthogeodesics are length equivalent (in $\mathcal{T}(\Sigma)$) if they are length equivalent in $\mathcal{T}_b(\Sigma)$.\end{lem}\begin{proof}Suppose that two orthogeodesics $\eta_1$ and $\eta_2$ are length equivalent in $\mathcal{T}_b(\Sigma)$. We can assume that their endpoints are on two elements $\alpha$ and $\beta$ in $\partial{\Sigma}$. Take any $S \in \mathcal{T}_b(\Sigma)$, \end{proof}

An \textbf{e-orthoshape} is a set of orthogeodesics satisfying some equality conditions on their truncated lengths which hold for any pair $(S,\hat{S}) \in \hat{\mathcal{T}}(\Sigma)$. By this definition, a pair of length-equivalent orthogeodesics is an example of an e-orthoshape. We are interested in some types of e-orthoshapes which closely related to length-equivalent orthogeodesics. 
\begin{dfn}\label{e-}
$\bb{O}$ is an \textbf{e-ortho-isosceles-trapezoid/rectangle/kite/parallelogram/rhombus/square} if it is an ortho-isosceles-trapezoid/rectangle/kite/parallelogram/rhombus/square with respect to any pair $(S,\hat{S}) \in \hat{\mathcal{T}}(\Sigma)$.
\end{dfn}
Due to the definition of length equivalent orthogeodesics (See Definition \ref{dfn:lengthequivalentorthogeodesics}), Definition \ref{e-} is equivalent to the following one: 
\begin{dfn}\label{r+}
$\bb{O}$ is an \textbf{e-orth-isosceles-trapezoid/rectangle/kite/parallelogram/rhombus/square} if it is an ortho-isosceles-trapezoid/rectangle/kite/parallelogram/rhombus/square with respect to $S^\star{}$ for any $S \in \mathcal{T}(\Sigma)$.
\end{dfn}

%\textbf{Remark.} By considering (doubly) truncated orthogeodesics and their lambda lengths, one can extend these definitions to a more general context when the boundary of surface consists cusps and simple closed geodesics. 
\section{Notation}\label{Notations} Since the following notations will be used throughout this paper, we put them in a separate section so readers can revisit them whenever they get confused. Let $X$ be the shortest geodesic arc from $A$ to $B$. %To avoid long expressions in several formulae in this paper, 
\begin{enumerate}
\item If $A$ is a geodesic and $B$ is a geodesic/point in the hyperbolic plane, we use 

\begin{enumerate}
\item $X$ (or $AB$) for $\cosh(\ell_\mathbb{H}(X))$,  \label{3.1a}
\item{}$\overline{X}$ (or $\overline{AB}$) for $\cosh(\ell_\mathbb{H}(X)/2)$, \label{3.1b}
\item{}$\widetilde{X}$ (or $\widetilde{AB}$) for $\sinh(\ell_\mathbb{H}(X))$. \label{3.1c}
\end{enumerate}
\item If $A$ is a horocycle/geodesic/point and $B$ is a horocycle in the hyperbolic plane, we use

\begin{enumerate}
\item{}$X$ (or $AB$) for $\frac{1}{2}e^{d_\mathbb{H}(A,B)}$, \label{3.2a}
\item{}$\lambda (X)$ (or $\lambda (A,B)$) for $e^{\frac{1}{2}d_\mathbb{H}(A,B)}$ (Penner's lambda length), \label{3.2b}
\item{}$\overline{X}$ (or $\overline{AB}$) for $\frac{1}{2}e^{\frac{1}{2}d_\mathbb{H}(A,B)}$ (Half of Penner's lambda length). \label{3.2c}
\end{enumerate}
\item If $A$  and $B$ are two points in the hyperbolic plane, we use

\begin{enumerate}
\item{}$\overline{X}$ (or $\overline{AB}$) for $\sinh(\ell_\mathbb{H}(X)/2)$. \label{3.3a}
\end{enumerate}
\end{enumerate}

\section{Orthotree and identity relations}\label{section 1}
For simplicity, in this section, we present a combinatorial structure and identity relations of orthogeodesics on a hyperbolic surface with boundaries consisting of simple closed geodesics. With suitable choices of mixed Ptolemy relations in Lemma \ref{mix}, the ideas also apply to any arbitrary hyperbolic surface with boundaries that include cusps and simple closed geodesics.

%Let $S$ be an orientable hyperbolic surface with boundary $\partial{S}$ consisting of simple closed geodesics and/or cusps. Let $H$ be an orthotriangle on $S$. We are always able to add more simple orthogeodesics to obtain an orthotriangulation if the surface is of finite type. In the case of infinite type surfaces, by \cite{HugoAlan21}, in order to get an orthotriangulation, the three orthogeodesics in $\partial{H}$ need to intersect any simple closed geodesic a finite number of times.

%In this context, any orthotriangulation on $S$ is a hexagonal decomposition. Let $H$ be an embedded orthotriangle on $S$. Thus, $H$ is a hexagon with its boundary $\partial{H}$ consisting of three simple orthogeodesics and three geodesic subsegments on $\partial S$. Let $\alpha$ be one of the three subsegments. Let $\eta_l$ and $\eta_r$ be two simple orthogeodesics on $\partial{H}$ adjacent to $\alpha$. Let $\mathcal{O}_\alpha$ be the set of oriented orthogeodesics starting at $\alpha$.

\subsection{Orthotree}

Let $S$ be an orientable hyperbolic surface with boundary $\partial S$ consisting of simple closed geodesics. Let $\eta_{\frac{0}{1}}$ and $\eta_{\frac{1}{1}}$ be two oriented orthogeodesics starting from a simple closed geodesic $\alpha$ at the boundary of $S$. The starting points of these two orthogeodesics divide $\alpha$ into two open subsegments, namely $\alpha_1$ and $\alpha_2$. Suppose that $\alpha_1 \neq \varnothing$. Denote by $\mathcal{O}_{\alpha_1}$ the set of oriented orthogeodesics starting from $\alpha_1$. The orientation of $S$ gives an order on $\mathcal{O}_{\alpha_1}$ independent of the hyperbolic structure. In particular, the order on $\mathcal{O}_{\alpha_1}$  is defined by the order of the starting points of $\mathcal{O}_{\alpha_1}$ on $\alpha_1$. By that, if $\eta_x$, $\eta_y$ and $\eta_z$ are three elements in $\mathcal{O}_{\alpha_1}$ and the starting point of $\eta_y$ lies between $\eta_x$ and $\eta_z$ with respect to the subsegment $\alpha_1$, we say $\eta_y$ lies between $\eta_x$ and $\eta_z$.

\subsubsection{Labelling orthogeodesics by rational numbers}\label{Orsa}

Let $\mathcal{B}$ be an orthobasis on $S$. Suppose that $\eta_{\frac{0}{1}}, \eta_{\frac{1}{1}}$, and $\alpha_1$ defined above are three sides of an orthotriangle such that the third orthogeodesic in this orthotriangle is an element in $\mathcal{B}$. In other words, $\eta_{\frac{0}{1}}, \eta_{\frac{1}{1}}$ are neighbors with respect to $\mathcal{B}$ (see Definition \ref{dfn:neighboringorthogeodesics}). Then we have the following lemma:

\begin{lem}\label{lem:labeling}
    There is an order-preserving bijection between $\mathcal{O}_{\alpha_1} \sqcup \{\eta_{\frac{0}{1}} \} \sqcup \{\eta_{\frac{1}{1}}\}$ and $\mathbb{Q}\cap[0,1]$.
\end{lem}
\begin{proof}
We recall that a Farey pair $(\frac{p}{q},\frac{m}{n})$ is a pair of two reduced fractions $\frac{p}{q}$ and $\frac{m}{n}$ such that $\frac{0}{1}\leq \frac{p}{q} < \frac{m}{n} \leq \frac{1}{1}$ and $qm-pn=1$. Their Farey sum is defined as $\frac{p}{q} \oplus \frac{m}{n}:=\frac{p+m}{q+n}$. Thus $\frac{p}{q} < \frac{p}{q} \oplus \frac{m}{n} < \frac{m}{n}$. Furthermore, $(\frac{p}{q},\frac{p}{q} \oplus \frac{m}{n})$ and $(\frac{p}{q} \oplus \frac{m}{n},\frac{m}{n})$ become two other Farey pairs.

For any element $X\in \mathcal{O}_{\alpha_1}$, we denote by $N(X)$ the number of intersections of $X$ with the orthobasis $\mathcal{B}$. We define a map $\Psi_1$ from $\mathbb{Q}\cap[0,1]$ to $\mathcal{O}_{\alpha_1}\sqcup (\eta_{\frac{0}{1}}\sqcup \eta_{\frac{1}{1}})$ as follows:

\underline{{\it Step 1:}} Take an element $X$ in $\mathcal{O}_{\alpha_1}$ such that $N(X)$ is minimal. We name it by $\eta_{\frac{0}{1}\oplus \frac{1}{1}}$ or equivalently $\eta_{\frac{1}{2}}$. The starting point of this element divides $\alpha_1$ into two disjoint open subsegments. By abuse of notation, we denote these two subsegments by  $(\eta_{\frac{0}{1}},\eta_{\frac{1}{2}})$ and $(\eta_{\frac{1}{2}},\eta_{\frac{1}{1}})$.

\underline{{\it Step $k\geq 2$:}}  $\alpha_1$ is divided into $2^{k-1}$ disjoint open subsegments of the form $(\eta_{\frac{p_1}{q_1}},\eta_{\frac{p_2}{q_2}})$ where $(\frac{p_1}{q_1},\frac{p_2}{q_2})$ is a Farey pair. Denote by $\mathcal{O}_{\alpha_1}(\eta_{\frac{p_1}{q_1}},\eta_{\frac{p_2}{q_2}})$ the set of oriented orthogeodesics in $\mathcal{O}_{\alpha_1}$ starting from the subsegment $(\eta_{\frac{p_1}{q_1}},\eta_{\frac{p_2}{q_2}})$. Apply step 1 for each of these subsegments, one obtains $2^{k}$ new disjoint subsegments.

By this construction, we exhaust all elements in $\mathcal{O}_{\alpha_1}$. Thus, the above labeling gives an order-preserving bijection:
$$\mathring{\Psi}_1: \mathbb{Q}\cap(0,1) \to \mathcal{O}_{\alpha_1}.$$
This map can be extended naturally to an order-preserving bijection:
$$\Psi_1:\mathbb{Q}\cap[0,1] \to \mathcal{O}_{\alpha_1} \sqcup \{\eta_{\frac{0}{1}} \} \sqcup \{\eta_{\frac{1}{1}}\}.$$

\end{proof}

\subsubsection{Labelling complementary regions of a tree by rational numbers}
Let $T_1$ be a planar rooted trivalent tree whose first vertex is of valence 1, and all other vertices are of valence 3. We visualize it by embedding $T_1$ to the lower half-plane with the root located at the origin (see Figure \ref{panttree} for an example). Let $E(T_1)$ be the set of edges of $T_1$. Each edge of the tree has two sides associated with two distinct complementary regions of the tree. Let $\Omega(T_1)$ be the set of complementary regions of the tree. We label elements in $\Omega(T_1)$ as follows: The two initial regions are labeled by fraction numbers $\frac{0}{1}$ and $\frac{1}{1}$. Since each vertex except the root has three regions surrounding it, so if we know the labels of two of them, we can label the third region by their Farey sum. Thus, one obtains a bijection 
$$\Psi_2:  \mathbb{Q}\cap [0,1] \to \Omega(T_1).$$ 
By this map, the order on $\mathbb{Q}\cap[0,1]$ then descends to an order on $\Omega(T_1)$.

\subsubsection{Labelling complementary regions of the tree $T_1$ by orthogeodesics}\label{labellingcomplementaryregions}

 From Lemma \ref{lem:labeling} and the discussion above, $\Psi:=\Psi_1 \circ \Psi_2^{-1}$ is an order-preserving bijection. Thus we have the following theorem:
 \begin{thm}\label{lem:labeling}
    There is an order-preserving bijection between $\mathcal{O}_{\alpha_1} \sqcup \{\eta_{\frac{0}{1}} \} \sqcup \{\eta_{\frac{1}{1}}\}$ and $\Omega(T_1)$.
\end{thm}

The map $\Psi$ gives us a label of $\Omega(T_1)$ by the set $\mathcal{O}_{\alpha_1} \sqcup \{\eta_{\frac{0}{1}} \} \sqcup \{\eta_{\frac{1}{1}}\}$.  Furthermore, if the orthotriangulation associated with $\mathcal{B}$ is standard (see Definition \ref{standard}), each oriented orthogeodesic in $\mathcal{O}_{\alpha_1}$ can also be encoded by its crossing sequence with the orthobasis $\mathcal{B}$. Note that two oriented orthogeodesics in $\mathcal{O}_{\alpha_1}$ are of the same crossing sequence (word) if and only if they are two different directions of an orthogeodesic with a symmetric word. 

\subsubsection{Labelling edges of the tree $T_1$ by orthogeodesics}\label{weightonedges}

Let $Y$ and $Z$ be two arbitrary complementary regions of $T_1$. Thus, $\Psi(Y)$ and $\Psi(Z)$ are two oriented orthogeodesics starting from $\alpha_1$ and define an open subsegment, $(\Psi(Y), \Psi(Z))$, of $\alpha_1$. These orthogeodesics, together with the subsegment $(\Psi(Y), \Psi(Z))$, define a unique orthogeodesic denoted by $F(\Psi(Y), \Psi(Z), \alpha_1)$, forming an orthotriangle on $S$. Note that $F_{\alpha_1}(\Psi(Y), \Psi(Z))$ and $(\Psi(Y), \Psi(Z))$ are opposite sides in this orthotriangle. If $x$ is an edge of the tree $T_1$ with two neighboring complementary regions $Y$ and $Z$, then we label $x$ with the orthogeodesic $F_{\alpha_1}(\Psi(Y), \Psi(Z))$.

\subsubsection{Coding orthogeodesics by cutting sequences}\label{codingorthogeodesics}
Due to the labeling of edges, each path of edges starting from the root is associated with a sequence of elements in the orthobasis (crossing sequence). We label the associated complementary region of the path by capitalizing the word formed from the crossing sequence. Figure \ref{Pants1} illustrates a pair of pants with a standard orthotriangulation. See Figure \ref{panttree} for a tree of orthogeodesics started from $\alpha_1$ on the pair of pants with labels on edges and complementary regions, and see Figure \ref{orthoex} for an example of the crossing sequence of an orthogeodesic on the pair of pants.

This coding method can be combined with a generalized version of Theorem \ref{thm:panttorusjumpintro} to calculate the trace of any orthogeodesic on a general surface.

\begin{figure}[htbp]
    \centering
    \subcaptionbox{%
        The standard orthotriangulation on a pair of pants.%
        \label{Pants1}%
    }{%
        \begin{minipage}[t]{0.45\linewidth}
            \centering
            % Include the first figure with labels
            \leavevmode \SetLabels
            \L(.49*.8) $\alpha_1$\\%
            \L(.49*.92) $\alpha_2$\\%
            \L(.72*.32) $\beta_1$\\%
            \L(.85*.08) $\beta_2$\\
            \L(.22*.32) $\gamma_1$\\
            \L(.1*.07) $\gamma_2$\\
            \L(.07*.55) $b$\\
            \L(.49*.15) $a$\\
            \L(.9*.55) $c$\\
            \endSetLabels
            \AffixLabels{\includegraphics[width=4cm,angle=0]{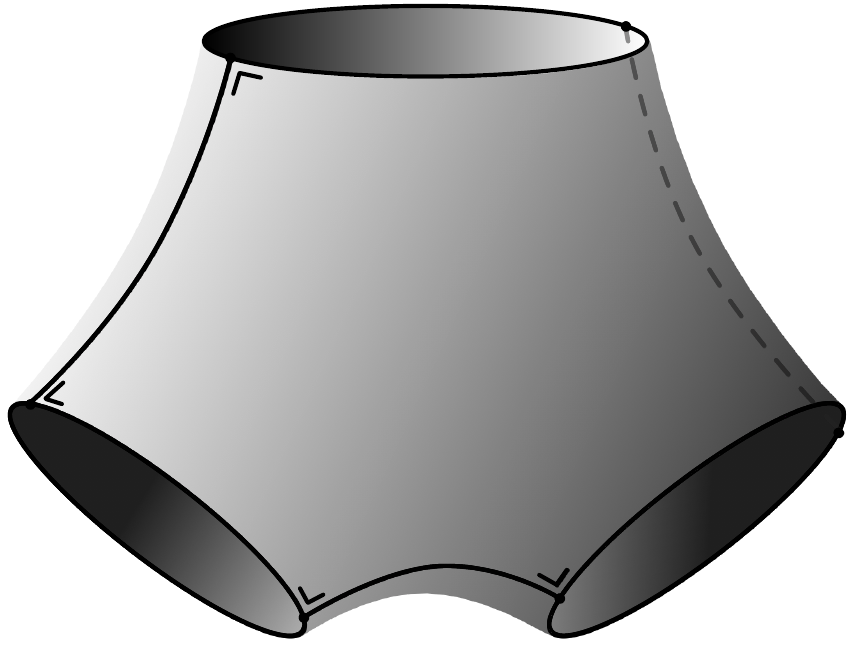}}
        \end{minipage}%
    }%
    \hfill
    \subcaptionbox{%
        Labeling the tree of oriented orthogeodesics starting at $\alpha_1$ on a pair of pants.%
        \label{panttree}%
    }{%
        \begin{minipage}[t]{0.45\linewidth}
            \centering
            % Include the second figure with caption
            \includegraphics[width=0.9\linewidth]{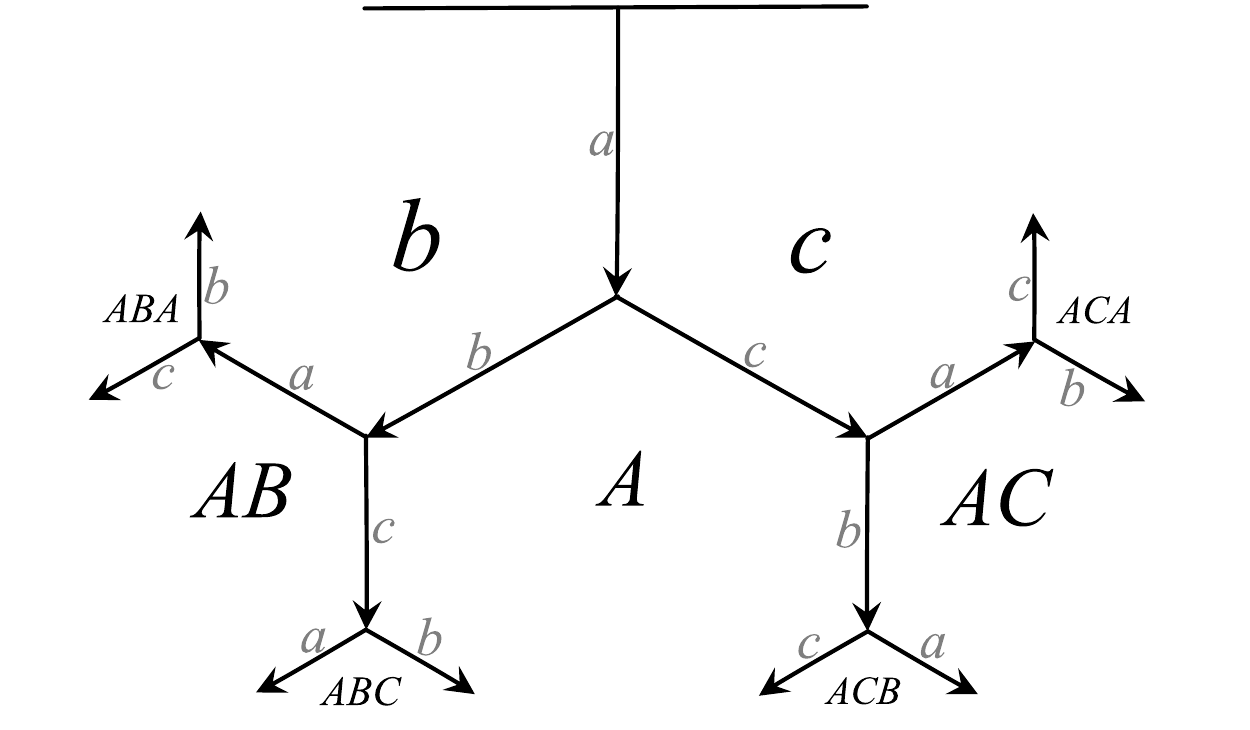}
        \end{minipage}%
    }%
    \caption{}
\end{figure}

%\textbf{Orthobasis:} Since $S$ is of finite type, an orthobasis of $S$ will be a set $\mathcal{E}:=\{e_1,e_2,...,e_{6g-6+3n}\}$ a collection of $6g-6+3n$ pairwise disjoint simple orthogeodesics. This set decomposes $S$ into a collection of $4g-4+2n$ right angled hexagons denoted by  $\mathcal{H}$. The set $\mathcal{E}$ cuts the boundary components of $S$ into a collection of $12g-12+6n$ subsegments denoted by $\mathcal{B}$. Note that each hexagon in $\mathcal{H}$ has 6 edges including three \textbf{non-boundary edges} (elements of $\mathcal{E}$) and three \textbf{boundary edges} (distinct elements of $\mathcal{B}$). Figure \ref{Pants1} is an illustration in the case of a pair of pants. We also note that, if $S$ is a cusped surface, the hexagonal decomposition becomes an ideal triangulation, and the hexagons in the decomposition become ideal triangles. By that reason, we may also call a hexagon of a hexagonal decomposition as an orthotriangle of an orthotriangulation.

 \begin{figure}[H]
%\ShowGrid
\leavevmode \SetLabels
\L(.29*.32) $\alpha_1$\\%
\endSetLabels
\begin{center}
\AffixLabels{\centerline{\includegraphics[width=7cm,angle=0]{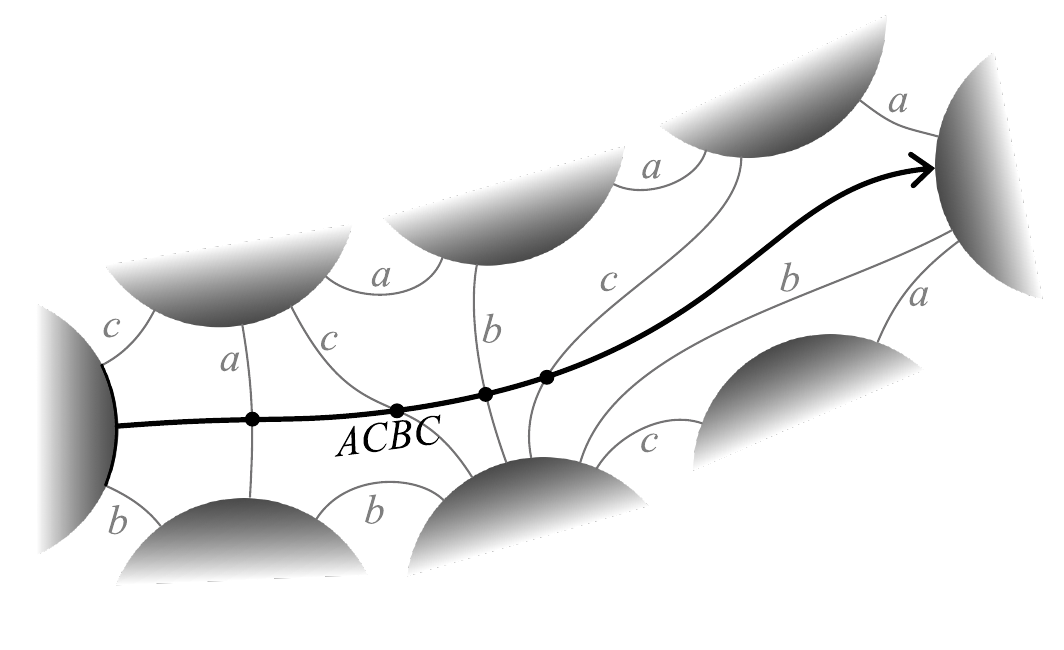}}}
\vspace{-24pt}
\end{center}
\caption{A lift of an oriented orthogeodesic $(ACBC)$ starting at $\alpha_1$.} \label{orthoex}
\end{figure}
Note that we have only defined a tree of oriented orthogeodesics starting at a subsegment of a simple closed geodesic $\alpha$ at the boundary of a surface. Since the orthobasis $\mathcal{B}$ divides $\alpha$ into a finite number of disjoint open subsegments, one can glue the roots of suitably labeled trees associated with all the subsegments in a cyclic order to get a planar rooted trivalent tree of all oriented orthogeodesics starting from $\alpha$.

\subsubsection{A more general way to label complementary regions of the tree $T_1$ by orthogeodesics}\label{coherentsubset}

A subset $G$ of $\mathcal{O}_{\alpha_1}$ is called \textbf{coherent} if, for any two elements of $G$, there are infinitely many other elements of $G$ between them. For instance, $G$ could be the set $\mathcal{O}_{\alpha_1}$ itself or the set of simple oriented orthogeodesics starting from $\alpha_1$ when $\alpha_1$ is chosen suitably.

In the proof of Lemma \ref{lem:labeling}, we defined an order-preserving bijection $\mathring{\Psi}_1$ from $\mathbb{Q} \cap (0,1)$ to $\mathcal{O}_{\alpha_1}$ with respect to a given orthobasis. We will now present a more general method to label elements in any coherent subset $G$ of $\mathcal{O}_{\alpha_1}$ using rational numbers between 0 and 1.

\underline{{\it Step 1:}} Take an arbitrary element in $G$ and name it by $\eta_{\frac{0}{1}\oplus \frac{1}{1}}$ or equivalently $\eta_{\frac{1}{2}}$. The starting point of this element divides $\alpha_1$ into two disjoint open subsegments, denoted by $(\eta_{\frac{0}{1}},\eta_{\frac{1}{2}})$ and $(\eta_{\frac{1}{2}},\eta_{\frac{1}{1}})$.

\underline{{\it Step $k\geq 2$:}} $\alpha_1$ is divided into $2^{k-1}$ disjoint open subsegments of the form $(\eta_{\frac{p_1}{q_1}},\eta_{\frac{p_2}{q_2}})$ where $(\frac{p_1}{q_1},\frac{p_2}{q_2})$ is a Farey pair. Denote by $G(\eta_{\frac{p_1}{q_1}},\eta_{\frac{p_2}{q_2}})$ the set of oriented orthogeodesics in $G$ starting from the subsegment $(\eta_{\frac{p_1}{q_1}},\eta_{\frac{p_2}{q_2}})$. Similarly to step 1, we take an arbitrary element in $G(\eta_{\frac{p_1}{q_1}},\eta_{\frac{p_2}{q_2}})$ and name it by $\eta_{\frac{p_1}{q_1}\oplus \frac{p_2}{q_2}}$ or equivalently $\eta_{\frac{p_1+p_2}{q_1+q_2}}$. We do it for all other subsegments and obtain $2^{k}$ new disjoint open subsegments.

The above labeling gives an order-preserving injection
$$\mathring{\Psi}'_1: \mathbb{Q}\cap(0,1) \to G.$$
This map can be extended naturally to an order-preserving injection
$$\Psi'_1:\mathbb{Q}\cap[0,1] \to G\sqcup \{\eta_{\frac{0}{1}}\}\sqcup \{\eta_{\frac{1}{1}}\}.$$

For each orthogeodesic \(X\), we denote \(h(X):=\frac{1}{2}\log\left(\frac{X+1}{X-1}\right)\) as the radius of the stable neighborhood of \(X\). In each step \(k \geq 1\), we have \(2^{k-1}\) disjoint open subsegments of \(\alpha_1\). Let \((X,Y)\) be one of these subsegments, and let \(\ell(X,Y)\) be the length of this subsegment. Let \(m(X,Y):=\ell(X,Y)-h(X)-h(Y)\) be the modified length of the segment \((X,Y)\). We denote by \(m_k\) the maximum of the modified lengths of the associated \(2^{k-1}\) disjoint subsegments in step \(k\). Then

\begin{lem} \label{mk}
The map \(\Psi'_1\) is an order-preserving bijection if \(m_k \to 0\) as \(k \to \infty\).
\end{lem}

\begin{proof}
Assume that \(X\) is an oriented orthogeodesic in \(G\) such that \(\Psi'_1(p/q) \neq X\) for all \(p/q \in \mathbb{Q} \cap [0,1]\). Let \((Y_k, Z_k)\) be the subsegment containing \(X\) at step \(k\). It is easy to see that \(m(Y_k, Z_k) \geq 2h(X)\) for all \(k \geq 1\), which is a contradiction to the fact that \(m_k \to 0\) as \(k \to \infty\).
\end{proof}

%The orientation of $S$ gives an order on $\mathcal{O}_\alpha$ independent of the hyperbolic structure. In particular, the order is defined by the order of the starting points of $\mathcal{O}_\alpha$ on $\alpha$. We order

%Let $\mathcal{TE}$ be the space of pairs $(T,e)$ where $T$ is an orthotriangulation of the universal cover of $S$ and $e$ is an edge of $T$. Fix a lift $\tilde{H}$ of the hexagon $H$. Let $\tilde{\alpha} \in \tilde{H}$ be a lift of $\alpha$. Each oriented orthogeodesic in $\mathcal{O}_\alpha$ can be lifted to an oriented simple orthogeodesic started from $\tilde{\alpha}$.  We are going to recall from \cite{LT17} three transformations $\mathcal{F}$, $\mathcal{R}$ and $\mathcal{L}$ on $\mathcal{TE}$. First $\mathcal{F}(T,e)=(T',e')$ where $T'$ is the orthotriangulation flipped at $e$, and $e'$ is the new corresponding edge so that $(e,e')$ is positively oriented. 

Combining with Lemma \ref{mk}, we obtain the following theorem:

\begin{thm}\label{bijection} If $m_k \to 0$ as $k \to \infty$, then the map $\Psi':=\Psi'_1 \circ \Psi_2^{-1}$ is an order-preserving bijection.
\end{thm}
Therefore, the map $\Psi'$ labels $\Omega(T_1)$ by the set $G \sqcup \{\eta_{\frac{0}{1}}\} \sqcup \{\eta_{\frac{1}{1}}\}$. One can check that when $G=\mathcal{O}_{\alpha_1}$ the map $\Psi$ defined in Section \ref{labellingcomplementaryregions} is an example of the map $\Psi'$ described in Theorem \ref{bijection}.

\subsection{Identity relations}

\subsubsection{Recursive formula and vertex relation}\label{recursiveformula}
Let $X,Y,Z$ be three complementary regions surrounding a vertex $v$ (not at the root) of the tree $T_1$. Let $x,y,z$ be three edges that intersect $X,Y,Z$ respectively at only $v$. Suppose that the set of edges of $T_1$ is oriented outwards from the root. If the endpoint of the oriented edge $x$ and the starting point of the oriented edges $y$ and $z$ coincide at $v$ (see Figure \ref{vertex}), by the Ptolemy relation of geodesics in Lemma \ref{recursive}, one can compute $X$ in term of $Y,Z,x,y,z$:
 $$X=\frac{(xy+z)Y+(xz+y)Z+\sqrt{(x^2+y^2+z^2+2xyz-1)(x^2+Y^2+Z^2+2xYZ-1)}}{x^2-1}.$$
 \vspace{-24pt}
  \begin{figure}[H]
%\ShowGrid
\leavevmode \SetLabels
\L(.49*.65) $x$\\%
\L(.44*.2) $y$\\%
\L(.55*.2) $z$\\%
\L(.495*.25) $v$\\
\L(.49*.05) $X$\\
\L(.56*.5) $Y$\\
\L(.43*.5) $Z$\\
%\L(.49*.11) $\gamma$\\
%\L(.49*.6) $c$\\
\endSetLabels
\begin{center}
\AffixLabels{\centerline{\includegraphics[width=3cm,angle=0]{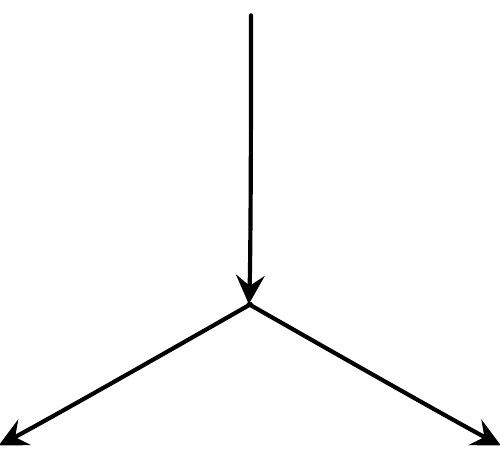}}}
\vspace{-24pt}
\end{center}
\caption{Positions of $X,Y,Z,x,y,z$.} \label{vertex}
\end{figure}

Note that by Corollary \ref{triangle}, we also have a \textbf{vertex relation} as follows:
\begin{equation*}
       (x^2-1)X^2+(y^2-1)Y^2+(z^2-1)Z^2-2(xy+z)XY-2(yz+x)YZ-2(xz+y)XZ
=x^2+y^2+z^2+2xyz-1 
\end{equation*}

%Hence each vertex of the tree is associated to a matrix. This matrix gives us a ternary quadratic form with determinant of a negative ``perfect square''.
\subsubsection{Isosceles trapezoid, rectangle, kite, parallelogram and edge relations}\label{Isosceles trapezoid and rectangle relations} 
For any two complementary regions \(U\) and \(V\) of the tree \(T_1\), we write \(UV\) as the trace of the orthogeodesic \(F_{\alpha_1}(\Psi(U), \Psi(V))\) as defined in Section \ref{weightonedges}. Let \(X, Y, Z,\) and \(T\) be four complementary regions. By Corollary \ref{all} in Section \ref{appendix}, we obtain the following relations:

\vspace{-3mm}
\begin{itemize}
\item[1.](Ortho-isosceles trapezoid): If $XT=YZ$, and $XZ=YT$, then: $\dfrac{X-Y}{\sqrt{XY+1}}=\dfrac{Z-T}{\sqrt{ZT+1}}.$ 
\item[2.](Rectangle): If $XT=YZ$, $XY=ZT$, and $XZ=YT$, then: $X+T=Y+Z.$
\item[3.](Kite): If $XY=XT$, and $ZY=ZT$, then: $T+Y=\dfrac{2(XZ.YZ+XY)}{(XZ)^2-1}X+\dfrac{2(XZ.XY+YZ)}{(XZ)^2-1}Z.$
\item[4.](Parallelogram): If $XY=ZT$, and $XT=YZ$, then: $X+Z=(Y+T)\sqrt{\dfrac{XZ-1}{YT-1}}.$
\end{itemize}
Furthermore, one obtains \textbf{edge relations} of orthogeodesics on certain special surfaces as in the following examples.

\textbf{Example 1: Edge relations on a pair of pants.} Denote by $\{a,b,c\}$ the standard orthobasis of a pair of pants, that is, the set of shortest simple orthogeodesics connecting two distinct boundary components. This orthobasis cuts each boundary geodesic component into two geodesic segments of equal length. The six resulting segments are denoted by $\alpha_1,\alpha_2,\beta_1,\beta_2,\gamma_1,\gamma_2$ as in Figure \ref{Pants1}. Let $T_1$ be the tree of oriented orthogeodesics starting from $\alpha_1$. As in Figure \ref{Pants1}, $b$ and $c$ are on the left and the right of the segment $\alpha_1$ and $a$ is opposite to $\alpha_1$. Each edge of $T_1$ is labeled by either $a$ or $b$ or $c$ following the grammar of the dual graph of the orthotriangulation. We also use words formed from the capital letters $A,B,C$ to label the complementary regions of $T_1$ except for the two initial regions. Figure \ref{panttree} is an illustration of the tree $T_1$.

Edge relations are special cases of orthokite relations when four edges of the orthokite are elements in the standard orthobasis of a pair of pants. Four regions surround each edge - except for the edge with an endpoint at the root. We choose arbitrarily an edge $x\in \{a,b,c\}$ with four surrounding regions $X_1,Y,Z,X_2$ with $Y \cap Z =x$. Due to the standard orthobasis of a pair of pants, $X_1Y=X_2Y=:z$ and $X_1Z=X_2Z=:y$, where $y,z$ are distinct elements in $\{a,b,c\}-\{x\}$. Thus one can denote the labels of $X_1\cap Y$, $X_1 \cap Z$, $X_2\cap Y$, $X_2 \cap Z$ by $z,y,z,y$ respectively (see e.g. Figure \ref{recur}). We have an edge relation of orthogeodesics on a pair of pants: 
\begin{equation}\label{Edgerelation1}
X_1+X_2=\frac{2(xy+z)}{x^2-1}Y+\frac{2(xz+y)}{x^2-1}Z.
\end{equation}

\textbf{Example 2: Edge relations on a one-holed torus.} Let $\{a,b,c\}$ be an arbitrary orthobasis of a one-holed torus, that is, the set of three non-intersecting simple orthogeodesics. This orthobasis cuts the boundary geodesic component of the one-holed torus into six geodesic segments. 
\begin{figure}[H]
\begin{minipage}[t]{0.5\textwidth}
  \centering
  % First figure
  %\ShowGrid
  \leavevmode \SetLabels
  \L(.5*.8) $Y$\\%
  \L(.5*.52) $x$\\%
  \L(.81*.25) $y$\\%
  \L(.81*.74) $z$\\%
  \L(.16*.25) $y$\\%
  \L(.16*.74) $z$\\%
  \L(.08*.5) $X_1$\\%
  \L(.5*.15) $Z$\\%
  \L(.87*.5) $X_2$\\%
  \endSetLabels
  \AffixLabels{\includegraphics[width=4cm,angle=0]{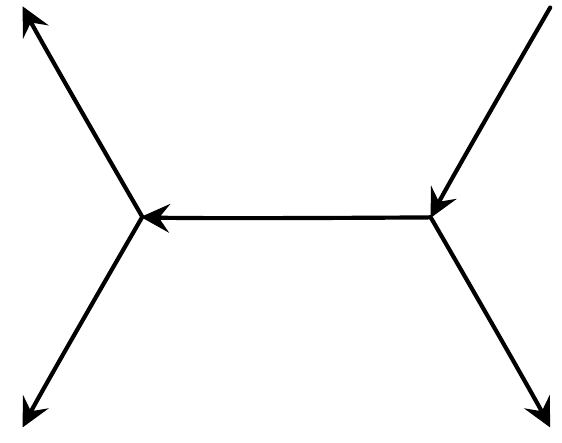}}
  % \vspace{-24pt}
  \caption{Positions of $X_1, X_2, Y, Z, x, y, z$ in the case of a pair of pants.}
  \label{recur}
\end{minipage}%
\begin{minipage}[t]{0.5\textwidth}
  \centering
  % Second figure
  %\ShowGrid
  \leavevmode \SetLabels
  \L(.5*.8) $Y$\\%
  \L(.5*.52) $x$\\%
  \L(.81*.25) $y$\\%
  \L(.81*.74) $z$\\%
  \L(.16*.25) $z$\\%
  \L(.16*.74) $y$\\%
  \L(.08*.5) $X_1$\\%
  \L(.5*.15) $Z$\\%
  \L(.87*.5) $X_2$\\%
  \endSetLabels
  \AffixLabels{\includegraphics[width=4cm,angle=0]{fig.edge.pdf}}
  % \vspace{-24pt}
  \caption{Positions of $X_1, X_2, Y, Z, x, y, z$ in the case of a one-holed torus.}
  \label{onehole}
\end{minipage}
\end{figure}
Let $T_1$ be the tree of oriented orthogeodesics starting from one of these six segments. Four regions surround each edge - except for the edge with an endpoint at the root. We choose arbitrarily an edge \( x \in \{a, b, c\} \) with four surrounding regions \( X_1, Y, Z, X_2 \) where \( Y \cap Z = x \). We have $X_1Y=X_2Z=:y$ and $X_1Z=X_2Y=:z$, where $y,z$ are distinct elements in $\{a,b,c\}-\{x\}$. Thus one can denote the labels of $X_1\cap Y$, $X_1 \cap Z$, $X_2\cap Y$, $X_2 \cap Z$ by $y,z,z,y$ respectively (see Figure \ref{onehole}). Then by the parallelogram relation, one has  $$X_1+X_2=(Y+Z)\sqrt{\frac{X_1X_2-1}{YZ-1}}.$$
Note that $X_1X_2$ is the trace of the simple orthogeodesic with label $X$ crossing the edge $x$. Thus one can use the recursive formula (Section \ref{recursiveformula}) to compute $X_1X_2$ in terms of $x,y,z$ as follows:
$$X_1X_2=\frac{(xy+z)y+(xz+y)z+x^2+y^2+z^2+2xyz-1}{x^2-1}=\frac{(y+z)^2}{x-1}+1.$$
Note that $YZ=x$, thus we have an edge relation of orthogeodesics on a one-holed torus:
\begin{equation}\label{Edgerelation2}X_1+X_2=\frac{y+z}{x-1}(Y+Z).
\end{equation}

% \comment{$\begin{bmatrix}
%     0&1&0\\
%     1&0&0\\
%     \frac{y+z}{x-1}&\frac{y+z}{x-1}&-1
    
% \end{bmatrix}\begin{bmatrix}
%     Y\\
%     Z\\
%     X_1
    
% \end{bmatrix}
% = \begin{bmatrix}
%     Z\\
%     Y\\
%     X_2
    
% \end{bmatrix}$}

%\begin{rk}
%This relation looks like  the Arithmetic progression (parallelogram) rule in the construction of a topograph of a quadratic form by Conway. Indeed, if we substitute $b$ and $c$ by $a$, and let $a$ goes to infinity, then we obtain the parallelogram relation:
%$$X_1+X_2=2(Y+Z).$$
%\end{rk}

\section{Applications}\label{applications}
In this section, we present some applications from the study of the tree structure on the set of orthogeodesics.
\subsection{Ortho-integral surfaces}\label{integralspectrum}
A hyperbolic surface is \textbf{ortho-integral} if the trace of each orthogeodesic is an integer. Denote by \(P(a,b,c)\) and \(T(a,b,c)\) respectively a pair of pants and a one-holed torus with an orthobasis of trace \((a,b,c)\). Using the vertex and edge relations, one can give conditions on pairs of pants and one-holed tori such that they are ortho-integral. 
\begin{thm}\label{Thm: integral surfaces}
(i) A pair of pants is ortho-integral if it admits an orthobasis $(a,b,c)$ belonging to the following set:  
$$(2,2,2),(3,3,3),(2,2,5),(3,3,7),(5,5,11).$$
(ii) A one-holed torus is ortho-integral if it admits an orthobasis $(a,b,c)$ belonging to the following set:   
$$ (2, 2, 2),(2, 2, 3),(2, 2, 5),(2, 3, 6),(2, 4, 4),(2, 4, 7),(2, 5, 8),(2, 7, 10),(2, 13, 16),$$
$$(3, 3, 3),(3, 3, 7),(3, 5, 5),(3, 5, 9),(3, 9, 13),(3, 17, 21),(4, 4, 5),(4, 5, 10),(4, 6, 6),(4, 11, 16),$$
$$(5, 5, 11),(5, 7, 13),(5, 13, 19),(6, 6, 9),(6, 36, 64),(7, 9, 9),(7, 17, 25),(8, 13, 22),(9, 11, 21),$$
$$(10, 10, 17),(10, 12, 12),(11, 49, 61),(13, 29, 43),(17, 19, 37),(19, 21, 21).$$
\end{thm}

\begin{proof}
We provide a proof of Theorem \ref{Thm: integral surfaces} for certain cases in Examples \ref{Example:1} and \ref{Example:2}. Other cases are either identical or simpler.

% \[
% (a^2-1)X^2 + (b^2-1)Y^2 + (c^2-1)Z^2 - 2(ab+c)XY - 2(bc+a)YZ - 2(ca+b)XZ = a^2 + b^2 + c^2 + 2abc - 1.
% \]

\begin{example}
    \label{Example:1} Let $P(2,2,5)$ be a pair of pants with an orthobasis $(2,2,5)$. The vertex relation (divided by 3) in this case is as follows.
\begin{equation}\label{equ: Diophantine}
    X^2+Y^2+8Z^2-6XY-8YZ-8ZX=24.
\end{equation}
In this case, the edge relations (refer to Equation \ref{Edgerelation1}) are Vieta's formulas. They are as follows: 
$$X_1+X_2 = 6Y+8Z, \,\,\,\,Y_1+Y_2 = 8Z+6X, \,\,\,\, Z_1+Z_2 = X+Y.$$ 
If we start with the solution \((X, Y, Z)\), the neighboring solutions will be 
$$\left(6Y+8Z-X,Y,Z\right),\,\,\,\, \left(X,8Z+6X-Y,Z\right), \,\,\,\, \left(X,Y,X+Y-Z\right).$$ 
This Vieta jumping corresponds to the action of the following group on the set of solutions of Equation \ref{equ: Diophantine}:
$$G_{P(2,2,5)}:=\left\langle f_X:=\begin{bmatrix}
    -1 & 6& 8\\
    0 & 1& 0\\
    0&0&1
\end{bmatrix},f_Y:=\begin{bmatrix}
    1 & 0& 0\\
    6 & -1& 8\\
    0&0&1
\end{bmatrix}, f_Z:=\begin{bmatrix}
    1 & 0 & 0\\
    0 & 1& 0\\
    1&1& -1
\end{bmatrix}\right \rangle .$$

\begin{figure}[h]
\centering
\begin{tikzpicture}[node distance=0.1cm and 0.5cm]
% Nodes
\node (a) at (0,0) {$( -1, 5, 2 )$};
\node (b) [right=of a] {$( 47, 5, 2 )$};
\node (bu) [above right=of b] {$( 47, 293, 2 )$};
\node (bd) [below right=of b] {$( 47, 5, 50 )$};
\node (buu) [above right=of bu] {$( 47, 293, 338 )\cdots$};
\node (bud) [right=of bu] {$(1727, 293, 2 )\cdots$};
\node (bdu) [right=of bd] {$(383, 5, 50 )\cdots$};
\node (bdd) [below right=of bd] {$(47, 677, 50 )\cdots$};

% Arrows
\draw[-] (a) --node[midway, above] {$f_X$} (b);
\draw[-] (b) --node[midway, above] {$f_Y$} (bu);
\draw[-] (b) --node[midway, above] {$f_Z$} (bd);
\draw[-] (bu) --node[midway, above] {$f_Z$} (buu);
\draw[-] (bu) --node[midway, below] {$f_X$} (bud);
\draw[-] (bd) --node[midway, above] {$f_X$} (bdu);
\draw[-] (bd) --node[midway, below] {$f_Y$} (bdd);

\draw[-] (a.150) to[out=120, in=150, looseness=8] node[midway, above left] {$f_Z$} (a.160);
\draw[-] (a.200) to[out=240, in=270, looseness=8] node[midway, below left] {$f_Y$} (a.210);

\end{tikzpicture}
\caption{Orbit of \((-1, 5, 2)\) under the action of \(G_{P(2,2,5)}\).}
\label{orbitpants}
\end{figure}

The traces of all orthogeodesics on \(P(2,2,5)\) are positive numbers in the \(G_{P(2,2,5)}\)-orbits of fundamental solutions \((-1, 5, 2)\), \((5, -1, 2)\), and \((2, 2, -1)\) of Equation \ref{equ: Diophantine}. Thus, it is obvious that \(P(2,2,5)\) is ortho-integral. There are three distinct orbits, each corresponding to one of the three boundary components of the pair of pants.  Figure \ref{orbitpants} illustrates the orbit of \((-1, 5, 2)\) under the action of \(G_{P(2,2,5)}\).

\end{example}

The following example is much more subtle. 

\begin{example}
    \label{Example:2} Let $T(3,17,21)$ be a torus with an orthobasis $(3,17,21)$. This torus is obtained from gluing two boundaries of the pair of pants with standard orthobasis $(3,3,19)$. The vertex relation (divided by $8$) in this case is as follows.
\begin{equation}\label{equ: Diophantinetorus}
    X^2+36Y^2+55Z^2-18XY-90YZ-20XZ=360
\end{equation}
The edge relations in this case (refer to Equation \ref{Edgerelation2}) are not Vieta's formulas. They are as follows:
\[ X_1 + X_2 = 19(Y + Z), \quad Y_1 + Y_2 = \frac{3}{2}(Z + X), \quad Z_1 + Z_2 = X + Y, \]
then combined with a \textit{swap} between two other entries (see Figure \ref{fig:swap}). 
\begin{figure}[H]
%\ShowGrid
\leavevmode \SetLabels
\L(.5*.8) $Y$\\%
\L(.5*.52) $3$\\%
\L(.57*.25) $17$\\%
\L(.579*.74) $21$\\
\L(.41*.25) $21$\\
\L(.4*.74) $17$\\
\L(0.27*.5) $19(Y+Z)-X$\\
\L(.5*.15) $Z$\\
\L(.65*.5) $X$\\
\endSetLabels
\begin{center}
\AffixLabels{\centerline{\includegraphics[width=3cm,angle=0]{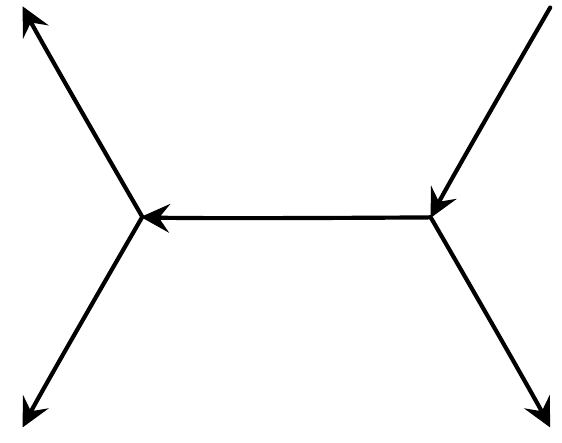}}}
\vspace{-24pt}
\end{center}
\caption{The jump $(X,Y,Z) \leftrightarrow (19(Y+Z)-X,Z,Y)$.} \label{fig:swap}
\end{figure}
In particular, if we start with the solution \((X, Y, Z)\), the neighboring solutions will be 
$$\left(19(Y+Z)-X,Z,Y\right),\,\,\,\, \left(Z,\frac{3}{2}(Z+X)-Y,X\right), \,\,\,\, \left(Y,X,X+Y-Z\right).$$ 
This type of root flipping corresponds to the action of the following group on the set of solutions of Equation \ref{equ: Diophantinetorus}:
$$G_{T(3,17,21)}:=\left\langle g_X:=\begin{bmatrix}
    -1 & 19& 19\\
    0 & 0& 1\\
    0&1&0
\end{bmatrix},g_Y:=\begin{bmatrix}
    0 & 0& 1\\
    3/2 & -1& 3/2\\
    1&0&0
\end{bmatrix}, g_Z:=\begin{bmatrix}
    0 & 1 & 0\\
    1 & 0& 0\\
    1&1& -1
\end{bmatrix}\right \rangle .$$
% \comment{Next paper: Is this group a thin group?}
The traces of all orthogeodesics on \( T(3,17,21) \) are positive numbers in the \( G_{T(3,17,21)} \)-orbit of the fundamental solution \( (-1, 21, 17) \) of Equation \ref{equ: Diophantinetorus}, illustrated in Figure \ref{fig:orbittorus}.

\begin{figure}[h]
\centering
\begin{tikzpicture}[node distance=0.1cm and 0.5cm, >={Stealth}]
% Nodes
\node (a) at (1.5,0) {$( -1, 21, 17 )$};
\node (aa) [right=of a] {$( 723, 17, 21 )$};
\node (aau) [above right=of aa] {$( 21, 1099, 723 )\cdots$};
\node (aad) [below right=of aa] {$( 17, 723, 719 )\cdots$};
\node (b) at (0,-1.5) {$(17, 3, -1)$};
\node (c) at (0,1.5) {$(21, -1, 3)$};
\node (bb) [left=of b] {$(3, 17, 21 )$};
\node (bbu) [above left=of bb] {$\cdots( 21, 19, 3 )$};
\node (bbd) [ left=of bb] {$\cdots( 719, 21, 17 )$};
\node (cc) [left=of c] {$(3, 37, 21 )$};
\node (ccu) [ left=of cc] {$\cdots( 1099, 21, 37 )$};
\node (ccd) [below left=of cc] {$\cdots( 37, 3, 19 )$};
% Arrows
\draw[-] (a) -- node[midway, above] {$g_X$} (aa);
\draw[-] (aa) -- node[midway, above] {$g_Y$} (aau);
\draw[-] (aa) -- node[midway, above] {$g_Z$} (aad);
\draw[-] (a) -- node[midway, right] {$g_Y$} (b);
\draw[-] (a) -- node[midway, right] {$g_Z$} (c);
\draw[-] (b) -- node[midway, left] {$g_X$} (c);
\draw[-] (b) -- node[midway, above] {$g_Z$} (bb);
\draw[-] (c) -- node[midway, above] {$g_Y$} (cc);
\draw[-] (bb) -- node[midway, above] {$g_Y$} (bbu);
\draw[-] (bb) -- node[midway, below] {$g_X$} (bbd);
\draw[-] (cc) -- node[midway, above] {$g_X$} (ccu);
\draw[-] (cc) -- node[midway, below] {$g_Z$} (ccd);
\end{tikzpicture}
\caption{Orbit of \(( -1, 21, 17 )\) under the action of $G_{T(3,17,21)}$.}
\label{fig:orbittorus}
\end{figure}

It remains to prove that the $G_{T(3,17,21)}$-orbit of $( -1, 21, 17 )$ is in $\mathbb{Z}^3$. Indeed, we note that if $(X,Y,Z)$ is an integer solution of Equation \ref{equ: Diophantinetorus}, then $X^2+55Z^2$ must be even. It implies that $X+Z$ must be even, and $(Z,\frac{3}{2}(X+Z)-Y,X)$ is also an integer solution of Equation \ref{equ: Diophantinetorus}. Thus $T(3,17,21)$ is ortho-integral.
\end{example}
\end{proof}

Let \( S \) be a hyperbolic surface with non-empty, totally geodesic boundary and an orthobasis \(\mathcal{B}\). Two oriented orthogeodesics on \( S \) are \textit{(Farey) neighbors} (with respect to \(\mathcal{B}\)) if they start from the same boundary component and, together with an element of \(\mathcal{B}\) and the boundary components, form an immersed hexagon on \( S \) (see Definition \ref{dfn:neighboringorthogeodesics} for a more general version).

The following statement is generalized from Examples \ref{Example:1} and \ref{Example:2}. 

\begin{thm}\label{thm:panttorusjump}
\bi
\item[(i)] If $\textbf{v}=(X,Y,Z)^T$ is a vector of traces of three pairwise neighboring orthogeodesics on a pair of pants $P(a,b,c)$, then all of the traces of orthogeodesics starting from the same boundary component with $X,Y,Z$ on $P(a,b,c)$ are given by the coordinates of vectors in the orbit $G_{P(a,b,c)}\cdot \textbf{v}$, where $G_{P(a,b,c)}$ is a group
generated by
$$f_X=\begin{bmatrix}
    -1 & \frac{2(ab+c)}{a^2-1}& \frac{2(ac+b)}{a^2-1}\\
    0 & 1& 0\\
    0&0&1
\end{bmatrix},f_Y=\begin{bmatrix}
    1 & 0& 0\\
    \frac{2(ba+c)}{b^2-1} & -1& \frac{2(bc+a)}{b^2-1}\\
    0&0&1
\end{bmatrix},
f_Z=\begin{bmatrix}
    1 & 0 & 0\\
    0 & 1& 0\\
    \frac{2(ca+b)}{c^2-1}&\frac{2(cb+a)}{c^2-1}& -1
\end{bmatrix}.$$
\item[(ii)] If $\textbf{v}=(X,Y,Z)^T$ is a vector of traces of three pairwise neighboring orthogeodesics on a one-holed torus $T(a,b,c)$, then all of the traces of orthogeodesics on $T(a,b,c)$ are given by the coordinates of vectors in the orbit $G_{T(a,b,c)}\cdot \textbf{v}$, where $G_{T(a,b,c)}$ is a group generated by
$$g_X=\begin{bmatrix}
    -1 & \frac{b+c}{a-1}& \frac{b+c}{a-1}\\
    0 & 0& 1\\
    0&1&0
\end{bmatrix},g_Y=\begin{bmatrix}
    0 & 0& 1\\
    \frac{a+c}{b-1} & -1& \frac{a+c}{b-1}\\
    1&0&0
\end{bmatrix},
g_Z=\begin{bmatrix}
    0 & 1 & 0\\
    1 & 0& 0\\
    \frac{a+b}{c-1}&\frac{a+b}{c-1}& -1
\end{bmatrix}.$$
\ei
\end{thm}
% \textbf{Remark.} We can use the Vieta jump to generate solutions of Equation \ref{equ: Diophantinetorus}, but these solutions may not be integers. Thus, it is intriguing to determine if all the integer solutions obtained from the Vieta jump are identical to the set of all integer solutions obtained from the torus jump.

These above results tell us the following corollary:
\begin{cor}
If $(a,b,c)$ is an orthobasis of an ortho-integral torus or pair of pants, then the following Diophantine equation has infinitely many integral solutions:
$$(a^2-1)X^2+(b^2-1)Y^2+(c^2-1)Z^2-2(ab+c)XY-2(bc+a)YZ-2(ac+b)XZ = a^2+b^2+c^2+2abc-1.$$
\end{cor}

 \textbf{Remark.} We observe that there are four distinct orthobases on a pair of pants, whereas a one-holed torus possesses infinitely many distinct orthobases. Consequently, there are infinitely many Diophantine equations associated with a one-holed torus. Refer to Table \ref{table:2} for a list of some Diophantine equations related to ortho-integral tori and pairs of pants.

 \begin{longtable}{|c|c|}

\hline
\cellcolor{gray!30} \textbf{Orthobasis $(a,b,c)$}  & \cellcolor{gray!30} \textbf{Diophantine equation} \\
\hline
\endfirsthead

\hline
\cellcolor{gray!30} \textbf{Orthobasis $(a,b,c)$}  & \cellcolor{gray!30} \textbf{Diophantine equation} \\
\hline
\endhead

% \hline
% \endfoot

% \hline
% \endlastfoot

  \hline
  $(2,2,2)^{\star}$ & $X^2+Y^2+Z^2-4XY-4YZ-4XZ=9$  \\
  \hline
  $(2,2,3)$ & $3X^2+3Y^2+8Z^2-14XY-16YZ-16XZ=40$ \\
  \hline
  $(2,2,5)^{\star}$ & $X^2+Y^2+8Z^2-6XY-8YZ-8XZ=24$\\
  \hline
  $(2,3,6)$ & $3X^2+3Y^2+35Z^2-24XY-40YZ-30XZ=120$\\
  \hline
  $(2,4,4)$ & $X^2+5Y^2+5Z^2-8XY-12YZ-8XZ=33$\\
  \hline
  $(2,4,7)$ & $X^2+5Y^2+16Z^2-10XY-20YZ-12XZ=60$\\
  \hline
  $(2,5,8)$ & $X^2+8Y^2+21Z^2-12XY-28YZ-14XZ=84$\\
  \hline
  $(2,7,10)$ & $X^2+16Y^2+33Z^2-16XY-48YZ-18XZ=144$\\
  \hline
  $(2,3,16)$ & $X^2+56Y^2+85Z^2-28XY-140YZ-30XZ=420$\\

  \hline
  $(3,3,3)^{\star}$ & $X^2+Y^2+Z^2-3XY-3YZ-3XZ=10$\\
  \hline
  $(3,3,7)^{\star}$ & $X^2+Y^2+6Z^2-4XY-6YZ-6XZ=24$\\
  \hline
  $(3,5,5)$ & $X^2+3Y^2+3Z^2-5XY-7YZ-5XZ=26$\\
  \hline
  $(3,5,9)$ & $X^2+3Y^2+10Z^2-6XY-12YZ-8XZ=48$\\
  \hline
  $(3,9,13)$ & $X^2+10Y^2+21Z^2-10XY-30YZ-12XZ=120$\\
  \hline
  $(3,17,21)$ & $X^2+36Y^2+55Z^2-18XY-90YZ-20XZ=360$\\
  \hline
  $(4,4,5)$ & $5X^2+5Y^2+8Z^2-14XY-16YZ-16XZ=72$\\
  \hline
  $(4,5,10)$ & $5X^2+8Y^2+33Z^2-20XY-36YZ-30XZ=180$\\
  \hline
  $(4,6,6)$ & $3X^2+7Y^2+7Z^2-12XY-16YZ-12XZ=75$\\
  \hline
  $(4,11,16)$ & $X^2+8Y^2+17Z^2-8XY-24YZ-10XZ=120$\\
  \hline
  $(5,5,11)^{\star}$ & $X^2+Y^2+5Z^2-3XY-5YZ-5XZ=30$\\
  \hline
  $(5,7,13)$ & $X^2+2Y^2+7Z^2-4XY-8YZ-6XZ=48$\\
  \hline
  $(5,13,19)$ & $X^2+7Y^2+15Z^2-7XY-21YZ-9XZ=126$\\
  \hline
  $(6,6,9)$ & $7X^2+7Y^2+16Z^2-18XY-24YZ-24XZ=160$\\
  \hline
  $(6,36,64)$ & $X^2+37Y^2+117Z^2-16XY-132YZ-24XZ=945$\\
  \hline
  $(7,9,9)$ & $3X^2+5Y^2+5Z^2-9XY-11YZ-9XZ=84$\\
  \hline
  $(7,17,25)$ & $X^2+6Y^2+13Z^2-6XY-18YZ-8XZ=144$\\
  \hline
  $(8,13,22)$ & $3X^2+8Y^2+23Z^2-12XY-28YZ-18XZ=252$\\
  \hline
  $(9,11,21)$ & $2X^2+3Y^2+11Z^2-6XY-12YZ-10XZ=120$\\
  \hline
  $(10,10,17)$ & $11X^2+11Y^2+32Z^2-26XY-40YZ-40XZ=432$\\
  \hline
  $(10,12,12)$ & $9X^2+13Y^2+13Z^2-24XY-28YZ-24XZ=297$\\
  \hline
  $(11,49,61)$ & $X^2+20Y^2+31Z^2-10XY-50YZ-12XZ=600$\\
  \hline
  $(13,29,43)$ & $X^2+5Y^2+11Z^2-5XY-15YZ-7XZ=210$\\
  \hline
  $(17,19,37)$ & $4X^2+5Y^2+19Z^2-10XY-20YZ-18XZ=360$\\
  \hline
   $(19,21,21)$ & $9X^2+11Y^2+11Z^2-21XY-23YZ-21XZ=450$\\
  \hline
\captionsetup{justification=centering}
\caption{List of Diophantine equations associated with ortho-integral tori. The triples marked with $\star$ also correspond to ortho-integral pairs of pants. }
\label{table:2} 
\end{longtable}

Moreover, by gluing some of these ortho-integral building blocks together without twisting (see Figure \ref{fig:4holedsphere}), one can obtain other ortho-integral surfaces. For example:

\begin{figure}[h]
\centering
\begin{tikzpicture}[scale=0.45] % Adjust scale if necessary
  % First instance, rotated by 90 degrees
  \begin{scope}[rotate=-90]
    % Ellipses
    \draw (0,0) ellipse (.5 and .1);
    \draw (-1,-2) ellipse (.5 and .1);
    \draw (1,-2) ellipse (.5 and .1);
    
    % Curves
    \draw (-.5,0) to[out=-90,in=90] node[midway, above] {y} (-1.5,-2);
    \draw (.5,0) to[out=-90,in=90] node[midway, below] {y} (1.5,-2);
    \draw (-.5,-2) to[out=90,in=90] node[midway, left] {x} (.5,-2);
  \end{scope}
  
  % Second instance, rotated by -90 degrees
  \begin{scope}[shift={(0.3,0)}, rotate=90] % Shift right to separate instances
    % Ellipses
    \draw (0,0) ellipse (.5 and .1);
    \draw (-1,-2) ellipse (.5 and .1);
    \draw (1,-2) ellipse (.5 and .1);
    
    % Curves
    \draw (-.5,0) to[out=-90,in=90] node[midway, below] {y} (-1.5,-2);
    \draw (.5,0) to[out=-90,in=90] node[midway, above] {y} (1.5,-2);
    \draw (-.5,-2) to[out=90,in=90] node[midway, right] {x} (.5,-2);
  \end{scope}
  
\end{tikzpicture}
\captionsetup{justification=centering}
\caption{The four-holed sphere obtained by gluing two pairs of pants $P(y, y, x)$ at the boundary component opposite to the orthogeodesic with label $x$, without twisting.}
\label{fig:4holedsphere}
\end{figure}

\begin{cor}\label{333+333}
The four-holed sphere formed by gluing two copies of a pair of pants with standard orthobasis $(3,3,3)$, without twisting, is ortho-integral.

\end{cor}
 
\begin{proof}
We choose a standard orthobasis where two arbitrary neighboring hexagons form an ortho-parallelogram. By this choice, one can observe that this four-holed sphere is iso-orthospectral (up to multiplicity 2) to the one-holed torus $T(3,17,21)$. Thus the four-holed sphere is also ortho-integral.
\end{proof}
We use the notation $(y, y, x) + (x, y, y) = (x, 2y^2 - 1, x + 2y^2)$ to indicate that the four-holed sphere obtained by gluing two pairs of pants $P(y, y, x)$ at the boundary component opposite to the orthogeodesic with label $x$, without twisting, is iso-orthospectral (up to multiplicity 2) to the one-holed torus $T(x, 2y^2 - 1, x + 2y^2)$. From the proof of Corollary \ref{333+333}, we have other examples of ortho-integral four-holed spheres as follows.
\bi
\item $(2,7,10)=(2,2,2)+(2,2,2)$,
\item $(3,17,21)=(3,3,3)+(3,3,3)$,
\item $(5,7,13)=(2,2,5)+(5,2,2)$,
\item $(7,17,25)=(3,3,7)+(7,3,3)=(2,2,17)+(17,2,2)$,
\item  $(11,49,61)=(5,5,11)+(11,5,5)$,
\item $(17,19,37)=(3,3,19)+(19,3,3)$.
\ei

\subsection{Some infinite (dilogarithm) identities}\label{example}
% We now look at a new type of identity due to Bridgeman. Let $S$ be a hyperbolic surface of totally geodesic boundary, the Bridgeman identity \cite{bridgeman2011orthospectra} on $S$ is as follows:
% \begin{equation}\label{bridgeman}
% -\frac{\pi^2}{2} \chi(S)=\sum_{\eta}\mathcal{L}\left(\frac{1}{\cosh^2(\ell(\eta)/2)}\right),
% \end{equation}
% where the sum runs over the set of orthogeodesics on the surface and $\mathcal{L}$ is the Rogers dilogarithm function. 
Applying Theorem \ref{thm:panttorusjump}, we compute the ortho trace spectrum for certain special surfaces and then express Basmajian's identity and Bridgeman's identity (see Equation \ref{equ:Bas&Bri}) in each case.

\begin{example}[A pair of pants $P(3,3,3)$] Let \(\textbf{v}_1 = \begin{pmatrix} -1 \\ 3 \\ 3 \end{pmatrix}\), \(\textbf{v}_2 = \begin{pmatrix} 3 \\ -1 \\ 3 \end{pmatrix}\), and \(\textbf{v}_3 = \begin{pmatrix} 3 \\ 3 \\ -1 \end{pmatrix}\). Let
$$G_{P(3,3,3)}:=\left\langle f_X = \begin{bmatrix}
    -1 & 3& 3\\
    0 & 1& 0\\
    0&0&1
\end{bmatrix},f_Y = \begin{bmatrix}
    1 & 0& 0\\
    3 & -1& 3\\
    0&0&1
\end{bmatrix}, f_Z = \begin{bmatrix}
    1 & 0 & 0\\
    0 & 1& 0\\
    3&3& -1
\end{bmatrix}\right \rangle .$$
By Theorem \ref{thm:panttorusjump}, the orbits \(G_{P(3,3,3)} \cdot \textbf{v}_i\), where \(i \in \{1,2,3\}\), provide the spectrum of traces of orthogeodesics on \(P(3,3,3)\). Figure \ref{orbitpants333} illustrates the orbit \(G_{P(3,3,3)} \cdot \textbf{v}_1\), with bold numbers indicating elements in the trace spectrum of orthogeodesics.

\begin{figure}[h]
\centering
\begin{tikzpicture}[node distance=0.1cm and 0.5cm]
% Nodes
\node (a) at (0,0) {$( -1, 3, 3 )$};
\node (b) [right=of a] {$(\textbf{19}, 3, 3 )$};
\node (bu) [above right=of b] {$( 19, \textbf{63}, 3 )$};
\node (bd) [below right=of b] {$( 19, 3, \textbf{63} )$};
\node (buu) [above right=of bu] {$( 19, 63, \textbf{243})\cdots$};
\node (bud) [right=of bu] {$(\textbf{179}, 63, 3 )\cdots$};
\node (bdu) [right=of bd] {$(\textbf{179}, 3, 63 )\cdots$};
\node (bdd) [below right=of bd] {$(19, \textbf{243}, 63 )\cdots$};

% Arrows
\draw[-] (a) --node[midway, above] {$f_X$} (b);
\draw[-] (b) --node[midway, above] {$f_Y$} (bu);
\draw[-] (b) --node[midway, above] {$f_Z$} (bd);
\draw[-] (bu) --node[midway, above] {$f_Z$} (buu);
\draw[-] (bu) --node[midway, below] {$f_X$} (bud);
\draw[-] (bd) --node[midway, above] {$f_X$} (bdu);
\draw[-] (bd) --node[midway, below] {$f_Y$} (bdd);

\draw[-] (a.150) to[out=120, in=150, looseness=8] node[midway, above left] {$f_Z$} (a.160);
\draw[-] (a.200) to[out=240, in=270, looseness=8] node[midway, below left] {$f_Y$} (a.210);

\end{tikzpicture}
\caption{Orbit of \((-1, 3, 3)\) under the action of \(G_{P(3,3,3)}\).}
\label{orbitpants333}
\end{figure}

Thus, the trace spectrum of orthogeodesics on $P(3,3,3)$ is as follows:
$$ 3, 3, 3, 19, 19, 19, 63, 63, 63, 63, 63, 63, 179, 179, 179, 179, 179, 179, 243, 243, 243, 243, 243, 243, $$
$$483, 483, 483, 483, 483, 483, 723, 723, 723, 723, 723, 723,723, 723, 723, 723, 723, 723,$$
$$ 899, 899, 899, 899, 899, 899,1279, 1279, 1279, 1279, 1279, 1279, \ldots$$

Basmajian's identity:
$$\left(\frac{1+\sqrt{5}}{2}\right)^2=2\left(\frac{10}{9}\right)\left(\frac{32}{31}\right)^2\left(\frac{90}{89}\right)^2\left(\frac{122}{121}\right)^2\left(\frac{242}{241}\right)^2\left(\frac{362}{361}\right)^4\left(\frac{450}{449}\right)^2\cdots$$

Bridgeman's identity:
$$\frac{\pi^2}{2}=3\mathcal{L}\left( \frac{1}{2}\right)+3\mathcal{L}\left( \frac{1}{10}\right)+6\mathcal{L}\left( \frac{1}{32}\right)+6\mathcal{L}\left( \frac{1}{90}\right)+6\mathcal{L}\left( \frac{1}{122}\right)+6\mathcal{L}\left( \frac{1}{242}\right)+12\mathcal{L}\left( \frac{1}{362}\right)+\cdots$$
\end{example} 
Similarly, we have the following identities:
\begin{example}[A pair of pants $P(2,2,2)$]
Basmajian's identity:
$$\left(\frac{1+\sqrt{3}}{2}\right)^2=\left(\frac{3}{2}\right)\left(\frac{18}{16}\right)\left(\frac{75}{73}\right)^2\left(\frac{288}{286}\right)^2\left(\frac{363}{361}\right)^2\left(\frac{1083}{1081}\right)^2\left(\frac{1443}{1441}\right)^4\left(\frac{1728}{1726}\right)^2\cdots $$
Bridgeman's identity:
$$\frac{\pi^2}{2}=3\mathcal{L}\left( \frac{2}{3}\right)+3\mathcal{L}\left( \frac{2}{18}\right)+6\mathcal{L}\left( \frac{2}{75}\right)+6\mathcal{L}\left( \frac{2}{288}\right)+6\mathcal{L}\left( \frac{2}{363}\right)+6\mathcal{L}\left( \frac{2}{1083}\right)+12\mathcal{L}\left( \frac{2}{1443}\right)+\cdots$$

\end{example} 

Note that dilogarithm identities in these above examples differ from those in \cite{bridgeman2021dilogarithm, jaipong2023dilogarithm}. Their terms are arranged over the set of complementary regions of a trivalent tree which can be associated with a Farey sequence as in McShane's identity and other identities \cite{hines2020infinite,hu2014new,tan2008generalized} involving the set of simple closed geodesics on a once-punctured torus. Figure \ref{a=b=c} illustrates the two cases. 

\begin{center}
\begin{minipage}{\linewidth}
\begin{center}
\includegraphics[width=0.7\linewidth]{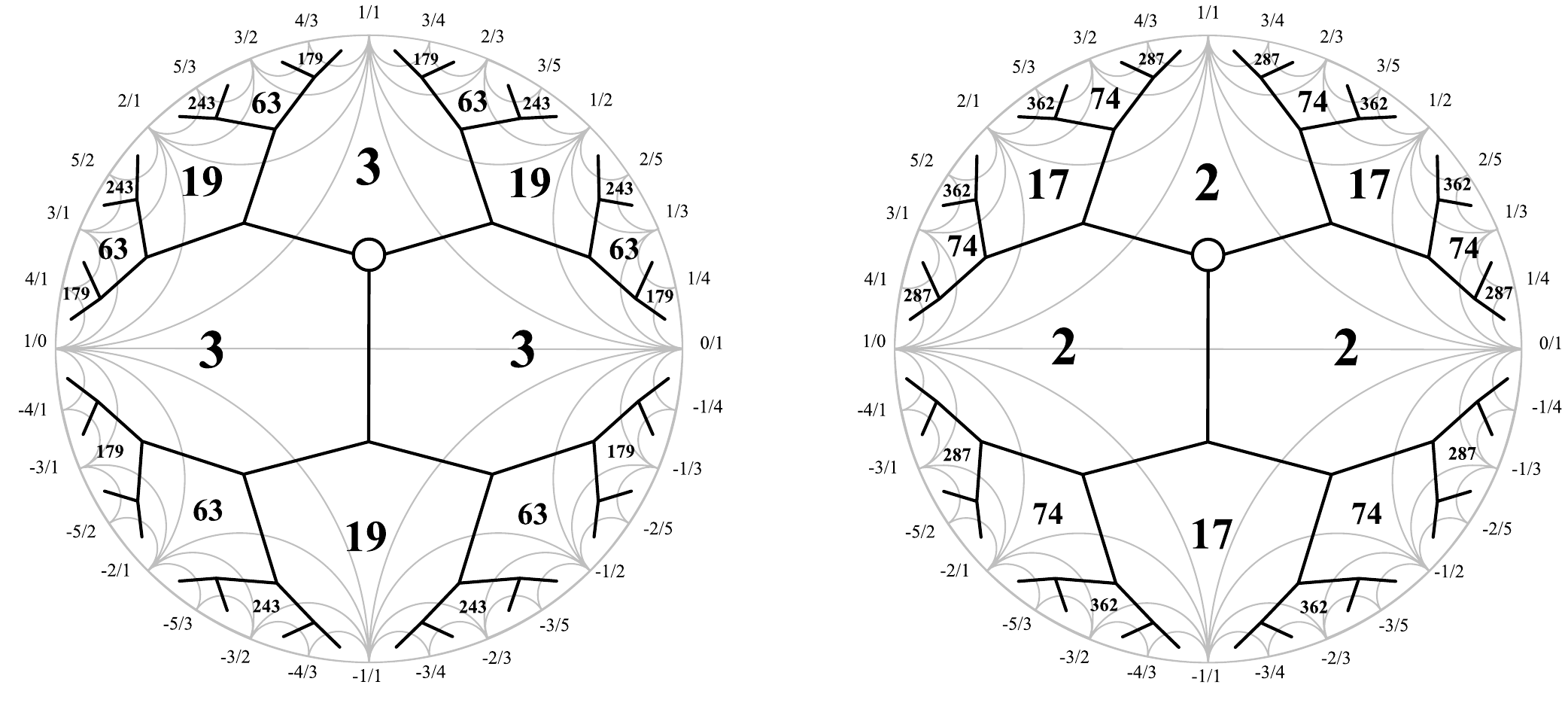}
\captionof{figure}{Examples: $(a,b,c)\in \{(2,2,2),(3,3,3)\}$.}
\label{a=b=c}
\end{center}
\end{minipage}
\end{center}

One can also investigate the set of one-holed tori with a regular orthobasis (see Proposition 4.3 in \cite{parlier2009simple}), then show that there are also two of them having the same ortho trace spectra as in the two examples above. 

\begin{example}[The one-holed torus $T(3,17,21)$]Refer to Example \ref{Example:2} for more details.

Basmajian's identity:
$$\left(\frac{1+\sqrt{5}}{2}\right)^2=\left(\frac{3}{2}\right)\left(\frac{10}{9}\right)\left(\frac{11}{10}\right)\left(\frac{19}{18}\right)\left(\frac{32}{31}\right)\left(\frac{41}{40}\right)\cdots $$
Bridgeman's identity:
$$\frac{\pi^2}{2}=\mathcal{L}\left( \frac{1}{2}\right)+\mathcal{L}\left( \frac{1}{9}\right)+2\mathcal{L}\left( \frac{1}{10}\right)+2\mathcal{L}\left( \frac{1}{11}\right)+2\mathcal{L}\left( \frac{1}{19}\right)+2\mathcal{L}\left( \frac{1}{32}\right)+2\mathcal{L}\left( \frac{1}{41}\right)+4\mathcal{L}\left( \frac{1}{90}\right)+2\mathcal{L}\left( \frac{1}{99}\right)+\cdots$$

\end{example} 
% \textbf{Example 4:} Description of Bridgeman's identity on the one-holed torus $T(17,19,37)$.
% $$\frac{\pi^2}{2}=\mathcal{L}\left( \frac{1}{9}\right)+\mathcal{L}\left( \frac{1}{10}\right)+2\mathcal{L}\left( \frac{1}{19}\right)+2\mathcal{L}\left( \frac{1}{72}\right)+2\mathcal{L}\left( \frac{1}{82}\right)+2\mathcal{L}\left( \frac{1}{90}\right)+2\mathcal{L}\left( \frac{1}{99}\right)+4\mathcal{L}\left( \frac{1}{199}\right)+\cdots$$

\subsection{A combinatorial proof of Basmajian's identity}\label{Bas}
Let $S$ be a hyperbolic surface of totally geodesic boundary, Basmajian's identity on $S$ is as follows:
$$ \ell(\partial{S})=\sum_{\eta}2\log(\coth(\ell(\eta)/2)),$$
where the sum runs over the set of orthogeodesics on the surface. This identity was proved by using elementary hyperbolic geometry and the fact that the limit set of a non-elementary second-kind Fuchsian group is of 1-dimensional measure zero. In this section, we will present a combinatorial proof of Basmajian's identity in the case of a hyperbolic surface with totally geodesic boundary. Our combinatorial setting will be as follows. 

Let $T_1$ be a rooted trivalent tree whose first vertex is of valence 1, and all other vertices are of valence 3. We visualize it by embedding $T_1$ in the lower half-plane with the root located at the origin (see Figure \ref{panttree}). Let $E(T_1)$ be the set of edges of $T_1$ where each edge is oriented outwards from the root. Let $\Omega(T_1)$ be the set of complementary regions of $T_1$. Let $\Phi : E(T_1) \sqcup \Omega(T_1) \to (1,\infty)$ be a map satisfying the following harmonic relation at any vertex except at the root of the tree:
\begin{equation}\label{harmonic}
    \arccosh\frac{x+YZ}{\sqrt{(Y^2-1)(Z^2-1)}} = \arccosh \frac{z+XY}{\sqrt{(X^2-1)(Y^2-1)}} + \arccosh \frac{y+XZ}{\sqrt{(X^2-1)(Z^2-1)}},
\end{equation}
where $x, y, z, X, Y, Z$ abbreviate $\Phi(x), \Phi(y), \Phi(z), \Phi(X), \Phi(Y), \Phi(Z)$ respectively, and they are as depicted in Figure \ref{vertex}.
Note that the geometric meaning of this harmonic relation comes from Lemma \ref{harmonic}, in which the function $\Phi$ is the cosh length function. By a slight abuse of notation, we will use $u$ as an abbreviation for $\Phi(u)$ for any $u \in E(T_1) \sqcup \Omega(T_1)$. We define a function $h$ on $\Omega(T_1)$ as follows:
$$h(X):=\arccosh\left(\frac{X}{\sqrt{X^2-1}}\right)=\frac{1}{2}\log\left(\frac{X+1}{X-1}\right).$$
A triple $(x,Y,Z)\in E(T_1)\times \Omega(T_1) \times \Omega(T_1)$ is called an edge region triple if $x=Y\cap Z$ (i.e. $Y,Z \in \Omega(T_1)$ are two neighboring complementary regions of $x \in E(T_1)$). We define a function $f$ on the set of edge region triples as follows:
$$f(x,Y,Z):=\arccosh\left(\frac{x+YZ}{\sqrt{(Y^2-1)(Z^2-1)}}\right).$$
Note that $\lim\limits_{Y\to \infty}f(y,X,Y)=\lim\limits_{Z\to \infty}f(z,X,Z)=h(X).$ The harmonic relation can be rewritten as:
 $$f(x,Y,Z)=f(y,X,Z)+f(z,X,Y),$$
for all $X,Y,Z,x,y,z$ as in Figure \ref{vertex}. Thus we have a Green formula:
$$\sum_{x \in C_n}f(x,Y,Z)=f(x_0,Y_0,Z_0),$$ where $(x_0,Y_0,Z_0)$ is the initial edge region triple (i.e., $x_0 \in E(T_1)$ is the initial edge from the root, $Y_0, Z_0 \in \Omega(T_1)$ are two neighboring complementary regions of $x_0$), and $C_n$ is the set of edges at combinatorial distance $n \in \mathbb{N}$ from the root. By Theorem \ref{lem:labeling}, we can state Basmajian's identity in a purely combinatorial manner as follows:

\begin{thm}\label{Basmaj}(Basmajian's identity for $T_1$)  If  $\sup\{\Phi(x) | x \in E(T_1) \}<\infty$, then
$$\sum_{X\in \Omega(T_1)}2h(X)=h(Y_0)+h(Z_0)+f(x_0,Y_0,Z_0).$$
or equivalently,
$$\sum_{X\in \Omega(T_1)}\log\left(\frac{X+1}{X-1}\right)=\frac{1}{2}\log\left(\frac{Y_0+1}{Y_0-1}\right)+\frac{1}{2}\log\left(\frac{Z_0+1}{Z_0-1}\right)+\arccosh\left(\frac{x_0+Y_0Z_0}{\sqrt{(Y_0^2-1)(Z_0^2-1)}}\right),$$
where $(x_0,Y_0,Z_0)$ is the initial edge region triple of the tree. Note that $x_0,X,Y_0,Z_0$ are the abbreviations of $\Phi(x_0),\Phi(X),\Phi(Y_0),\Phi(Z_0)$ respectively.
\end{thm}
\begin{proof}
The proof will follow closely to that of Bowditch for the McShane's identity. Firstly, it is easy to show that $f(x,Y,Z) \geq h(Y)+h(Z)$ for any edge region triple $(x,Y,Z)$. Then by using the Green formula, one obtains:
\begin{equation}\label{eq01}\sum_{X\in \Omega(T_1)}2h(X) \leq h(Y_0)+h(Z_0)+f(x_0,Y_0,Z_0).\end{equation}
Let $\mu := \sup\{\Phi(x) | x \in E(T_1) \}$, then we will show that
\begin{equation}\label{eq02}f(x,Y,Z) \leq h(Y)+x.h(Z) \leq h(Y)+\mu.h(Z),\end{equation} for any edge region triple $(x,Y,Z)$. Indeed, we define 
$$F(x):= h(Y) + x.h(Z) - f(x,Y,Z),  $$ then compute $F'(x)$ and $F''(x)$, in particular
$$F''(x)=\frac{x+YZ}{(x^2+Y^2+Z^2+2xYZ-1)^{3/2}} >0,$$ for all $x,Y,Z >1$. It implies that $F'(x)>F'(1)$. Note that $$F'(x)=\log \left(\frac{Z+1}{\sqrt{Z^2-1}} \right)-\frac{1}{\sqrt{x^2+Y^2+Z^2+2xYZ-1}},$$ so $$F'(x)>F'(1)=\frac{1}{2}\log \left(\frac{Z+1}{Z-1} \right)-\frac{1}{Y+Z}>\frac{1}{2}\log \left(\frac{Z+1}{Z-1} \right)-\frac{1}{Z+1}>0,$$ for all $x,Y,Z>1$. Hence $F(x)>F(1)=0$.

By combining Inequality \ref{eq02} and the Green formula we have:
$$f(x_0,Y_0,Z_0)=\sum_{x \in C_n}f(x,Y,Z)\leq \sum_{X\in \Omega_n}2h(X) -h(Y_0)-h(Z_0)+2\mu \sum_{X\in \Omega_{n+1}\backslash\Omega_{n}}2h(X),$$ for all $n\in \mathbb{N}$ where $\Omega_n$ is the set of complementary regions at combinatorial distance at most $n$ from the root. Let $n \to \infty$. Note that $\sum\limits_{X\in \Omega_{n+1} \backslash\Omega_{n}}2h(X)$ tends to 0 due to Inequality \ref{eq01}. One has \begin{equation}\label{eq03}f(x_0,Y_0,Z_0)\leq \sum_{X\in \Omega(T_1)}2h(X) -h(Y_0)-h(Z_0).\end{equation}
Finally, Basmajian's identity follows from Inequalities \ref{eq01} and \ref{eq03}.
\end{proof}
As a consequence, one can express a combinatorial form of Basmajian's identity for the set of oriented orthogeodesics starting from a simple closed geodesic, say $\alpha$, at the boundary of a hyperbolic surface, assuming that $\alpha$ is divided into $n$ subsegments by an orthobasis. Note that the finite sum on the right-hand side of Equation \ref{Basmajian1} is a combinatorial form of the length of $\alpha$.
\begin{cor}(Basmajian's identity for $T_n$)
Let $T_n$ be a rooted trivalent tree with $n$ edges starting from the root. If  $\sup\{\Phi(x) | x \in E(T_n) \}<\infty$, then 
\begin{equation}\label{Basmajian1}\sum_{X\in \Omega(T_n)}\log\left(\frac{X+1}{X-1}\right)=\sum_{k=1}^n \arccosh\left(\frac{x_{0,k}+Y_{0,k}Z_{0,k}}{\sqrt{(Y_{0,k}^2-1)(Z_{0,k}^2-1)}}\right).\end{equation}
where $(x_{0,k},Y_{0,k},Z_{0,k})$'s are edge region triples at the root of the tree. Note that $Z_{0,k}=Y_{0,k+1}$, with $k\in\{1,2,...,n\}$, and  $Y_{0,n+1}:=Y_{0,1}$.
\end{cor}
\textbf{Remarks.} 
\bi
\item The condition $\sup\{\Phi(x) \mid x \in E\} < \infty$ always holds for finite-type surfaces since the edges are labeled by a finite set, an orthobasis. If the surface is of infinite type with an unbounded hexagonal decomposition, i.e., $\sup\{\Phi(x) \mid x \in E\} = \infty$, then the proof in Theorem \ref{Basmaj} no longer applies. However, it does apply to the flute surface since there exists a bounded hexagonal decomposition for such a surface.

\item By adapting suitable harmonic relations (see the remark after Lemma \ref{computem}), one can extend this result to the general case in which the boundary of the surface consists of cusps and at least one simple closed geodesic. We suspect that this method can be generalized to higher dimensions.

\item In the proof of Theorem \ref{Basmaj}, we need two necessary inequalities: $h(Y)+h(Z) \leq f(x,Y,Z)$ and $f(x,Y,Z)\leq h(Y)+x.h(Z).$ It is not difficult to see the geometric meaning of the first inequality. However, the second one seems unnatural.
\ei

\section{A construction of e-orthoshapes using reflection involutions}\label{reflection}
 In this section, we will use reflection involutions on immersed pairs of pants to construct a family of e-ortho-isosceles-trapezoids, e-orthorectangles, and e-orthokites on a topological surface with punctures and negative Euler characteristic.  
 
 Let $P$ be a pair of pants with boundary components $o_1,o_2,o_3$. Let $a,b,c,A,B,C$ be simple orthogeodesics connecting two elements in each of pairs $(o_2,o_3)$, $(o_1,o_3)$, $(o_1,o_2)$, $(o_1,o_1)$, $(o_2,o_2)$, $(o_3,o_3)$ respectively. The three simple orthogeodesics $a,b,c$ are also called the \textbf{seams} of $P$. Let $r$ be the unique orientation-reversing isometry over $P$ which fixes $a,b,c$ pointwisely. 
 
 If $\ell(o_1)$, $\ell(o_2)$ and $\ell(o_3)$ are pairwise different, then the isometry group of $P$ is $\{id,r\}$ with $r^2=id$.  
 
 Let $\{i,j,k\}=\{1,2,3\}$. If $\ell(o_i)=\ell(o_j)$, let $r_k$ be the unique orientation-reserving isometry over $P$ which fixes $o_k$ and interchanges $o_i$ with $o_j$. The set of fixed points of each of isometries $r_1,r_2,r_3$ forms each of orthogeodesics $A,B,C$ respectively.
 
 If $\ell(o_i)=\ell(o_j) \neq \ell(o_k)$, the isometry group of $P$ will be $\{id, r_k, r, r\circ r_k\}$ with $r^2=r_k^2=(r\circ r_k)^2=id$ and $r\circ r_k=r_k\circ r$. 

 A \textbf{reflection involution} $i$ on $P$ is an orientation-reversing isometry on $P$ that fixes a simple orthogeodesic pointwisely (the ``symmetry axis'' of $i$). For instance, $r$ is a reflection involution on $P$ with symmetry axes $a,b,c$. If $\ell(o_i)=\ell(o_j)$, then $r_k$ is another reflection involution on $P$. 
 
 For a comprehensive treatment of pairs of pants, refer to chapter 3 of Buser's book \cite{buser2010geometry}.

Let $\Sigma$ be an oriented topological surface with punctures and negative Euler characteristic. A point in $\mathcal{T}(\Sigma)$ can be represented by a hyperbolic surface, namely $S$, with $\partial{S}$ consisting of simple closed geodesics and/or cusps such that the interior of $S$ is homeomorphic to $\Sigma$. Our main result in this section is the following theorem:
\begin{thm}\label{thm:reflection}
The involution reflections on an immersed pair of pants yield infinitely many e-ortho-isosceles-trapezoids, e-orthorectangles, and e-orthokites on \( S \). However, there are no e-orthosquares on \( S \).
\end{thm}

\begin{proof}

The key point is to consider an orthogeodesic $\gamma$ and an immersed pair of pants $P$ with $\gamma$ as one of its seams. By lifting this pair of pants $P$ to the universal cover (identified with $\mathbb{H}$), we obtain a fundamental domain for $P$ consisting of two right-angled hexagons with a reflection symmetry $r$ about the lift of $\gamma$. The connected component of the lift of $P$ containing this fundamental domain in $\mathbb{H}$ (denoted $\overline{P}$) will also be symmetric with respect to this reflection and is tessellated by isometric right-angled hexagons.

Considering any orthogeodesic in $\overline{P}$ (a geodesic arc connecting two boundaries of $\overline{P}$, which are lifts of boundary components of $S$), its reflection gives rise to the desired e-orthoshapes. It can be seen that these e-orthoshapes can be chosen to project to infinitely many different e-orthoshapes. In the case of e-orthorectangles, by starting with an orthogeodesic from a boundary component to itself, one can show that $\overline{P}$ is symmetric with respect to reflections on a pair of orthogonal geodesics.

Now we prove that there are no e-orthosquares on \( S \) by contradiction. Suppose that there exists $\mathbb{O}:=XYZT$ an e-orthosquare on $S$. We consider a hyperbolic structure $S^{*}\in \mathcal{T}(\Sigma)$ which satisfies:
\vspace{-3mm}
\begin{itemize}
\item $S^{*}$ is a cusped surface.
\item There is a truncated orthobasis (i.e. a decorated ideal triangulation) on $S^{*}$ including only orthogeodesics of lambda length 1.
\end{itemize} 
Note that if we lift this decorated ideal triangulation of $S^*$ to $\mathbb{H}$, we obtain the Farey tessellation with Ford circles. Also note that by Penner's Ptolemy relation, the tree of lambda lengths of orthogeodesics on $S^*$ is the tree of the Fibonacci function (see \cite{bowditch1998markoff} for a definition). Thus the lambda length of any orthogeodesics on $S^*$ is always an integer (also see \cite{penner2023music} for a direct computation on the Farey decoration). It is worth mentioning that the surface \( S^{*} \) is actually isometric to \(\mathbb{H}/G\), where \( G \) is a subgroup of the modular group \(\mathrm{PSL}(2,\mathbb{Z})\).

 Again by Penner's Ptolemy relation (see Equation \ref{PennerPto} and Notation \ref{3.2b}): $$\lambda(X,Z)\lambda(Y,T)=\lambda(X,Y)\lambda(Z,T)+\lambda(X,T)\lambda(Y,Z).$$ Thus $\lambda(X,Z)=\lambda(X,Y) \sqrt{2}$ since $\mathbb{O}$ is an orthosquare on $S^{*}$. It contradicts the fact that the lambda length of any orthogeodesic on $S^*$ is always an integer.
\end{proof}

\textbf{Remarks.} The existence of these types of
e-orthoshapes, particularly 
e-orthorectangles, on surfaces greatly simplifies the inter-root flipping for certain subsets of triples of orthogeodesics. Notably, this explains why there are quadruples of orthogeodesics $(X_1, Y, Z, X_2)$ on surfaces that always satisfy the trace identity: $$X_1+X_2=Y+Z$$ regardless of the hyperbolic structure of the surface. This insight is also useful for identifying integral suborbits of triples of orthogeodesics.

%These 6 closed geodesics also ``form'' an ideal isosceles trapezoid on the universal cover of the surface. 
%\begin{rk*}Each orthogeodesic is contained in a finite collection of immersed pairs of pants. This collection has an inclusion relation. One can easily compute a lower bound on the number of length equivalent orthogeodesic by taking the reflection of the orthogeodesic in all element of the collection. \end{rk*}

\section{Appendix: Ptolemy relations and their generalizations}\label{appendix}
This section presents several relations between the distances from a geodesic or a horocycle to a finite set of geodesics and/or horocycles in \(\mathbb{H}\). Some of these relations resemble those between a point and a finite set of points in the Euclidean plane. For example: there are Ptolemy relations and in particular, the orthorectangle and ortho-isosceles trapezoid relations are identical to the relation of the distances from a point to four vertices of a rectangle and an isosceles trapezoid in the Euclidean plane. These relations can also be expressed in terms of cross-ratios. For hyperbolic trigonometry formulae used in this section, we refer the reader to \cite{buser2010geometry} (Chapter 2) and \cite{beardon2012geometry} (Chapter 7).

\subsection{Penner's Ptolemy relation}
Let $A,B,C,P$ be four disjoint horocycles in $\mathbb{H}$. Each of them divides $\mathbb{H}$ into two domains such that the other three horocycles lie in the same domain. Following the notation in Item \ref{3.2c} of Section \ref{Notations}, we denote $\overline{X}:=\frac{1}{2}e^{\frac{1}{2}d_\mathbb{H}(P,A)}$, $\overline{Y}:=\frac{1}{2}e^{\frac{1}{2}d_\mathbb{H}(P,B)}$, $\overline{Z}:=\frac{1}{2}e^{\frac{1}{2}d_\mathbb{H}(P,C)}$, $\overline{x}:=\frac{1}{2}e^{\frac{1}{2}d_\mathbb{H}(B,C)}$, $\overline{y}:=\frac{1}{2}e^{\frac{1}{2}d_\mathbb{H}(A,C)}$, $\overline{z}:=\frac{1}{2}e^{\frac{1}{2}d_\mathbb{H}(A,B)}$.
% \begin{center}
% \begin{minipage}{\linewidth}
% \leavevmode \SetLabels
% \L(.465*.34) $\alpha$\\%
% \L(.49*.345) $\beta$\\
% \endSetLabels
% \begin{center}
% \includegraphics[width=0.25\linewidth]{fig.horotri.pdf}
% \captionof{figure}{Penner's Ptolemy relation.}
% \label{horo}
% \end{center}
% \end{minipage}
% \end{center}
If $P,B,A,C$ are in a cyclic order as in Figure \ref{horo}, then the Penner's Ptolemy relation (see \cite{penner1987decorated,penner2012decorated}) is \begin{equation}\label{PennerPto}\overline{X}=\frac{\overline{y}\overline{Y}+\overline{z}\overline{Z}}{\overline{x}}.
\end{equation}
This equation is equivalent to the harmonic relation: 
$$\frac{\overline{x}}{\overline{Y}.\overline{Z}}=\frac{\overline{y}}{\overline{X}.\overline{Z}}+\frac{\overline{z}}{\overline{X}.\overline{Y}},$$ 
in which {\small$\frac{\overline{x}}{\overline{Y}.\overline{Z}}; \frac{\overline{y}}{\overline{X}.\overline{Z}}; \frac{\overline{z}}{\overline{X}.\overline{Y}}$} are respectively the lengths of segments on $P$ between $PB$ and $PC$; $PA$ and $PC$; $PA$ and $PB$. Now, let's consider the case where all horocycles are replaced by geodesics.

\subsection{Ptolemy relation of geodesics}
Let $A,B,C,P$ be four disjoint geodesics in $\mathbb{H}$. Each of them divides $\mathbb{H}$ into two domains such that the other three geodesics lie in the same domain. Following the notation in Item \ref{3.1a} of Section \ref{Notations}, we denote $X:=\cosh d_\mathbb{H}(P,A)$, $Y:=\cosh d_\mathbb{H}(P,B)$, $Z:=\cosh d_\mathbb{H}(P,C)$, $x:=\cosh d_\mathbb{H}(B,C)$, $y:=\cosh d_\mathbb{H}(A,C)$, $z:=\cosh d_\mathbb{H}(A,B)$.

\begin{figure}[htbp]
    \centering
    \begin{minipage}[t]{0.45\linewidth}
        \centering
        \setlength{\unitlength}{\linewidth}
        \begin{picture}(0,0)
        % \put(.2,.34){$\alpha$}
        % \put(.3,.345){$\beta$}
        \end{picture}
        \includegraphics[width=0.55\linewidth]{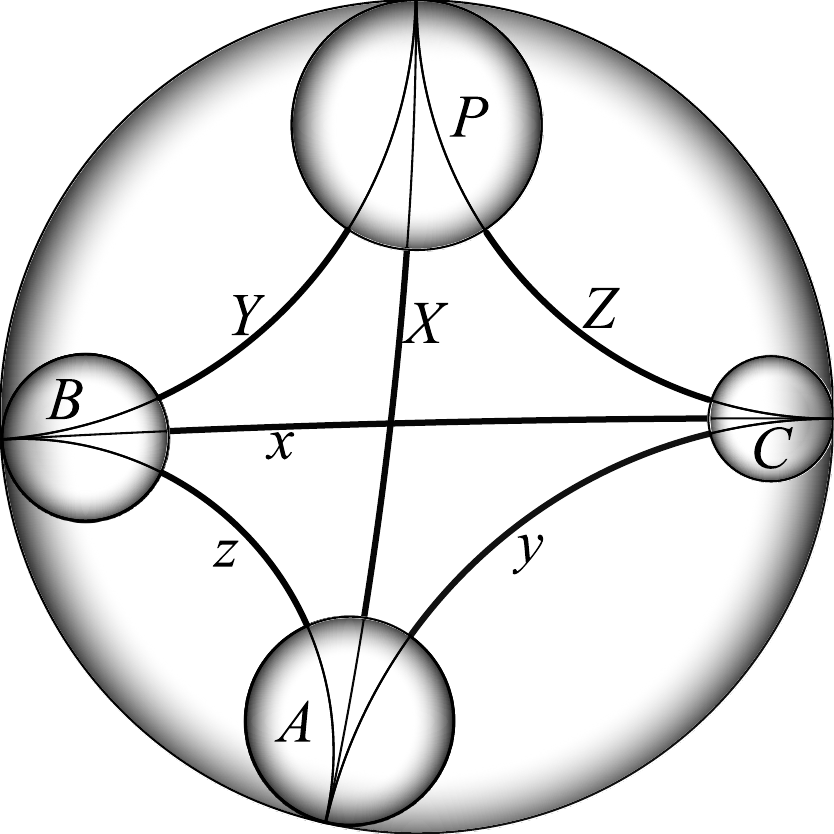}
        \caption{Penner's Ptolemy relation.}
        \label{horo} % Label for referencing
    \end{minipage}
    \hspace{0.05\linewidth} % Adjusted space
    \begin{minipage}[t]{0.45\linewidth}
        \centering
        \includegraphics[width=0.57\linewidth]{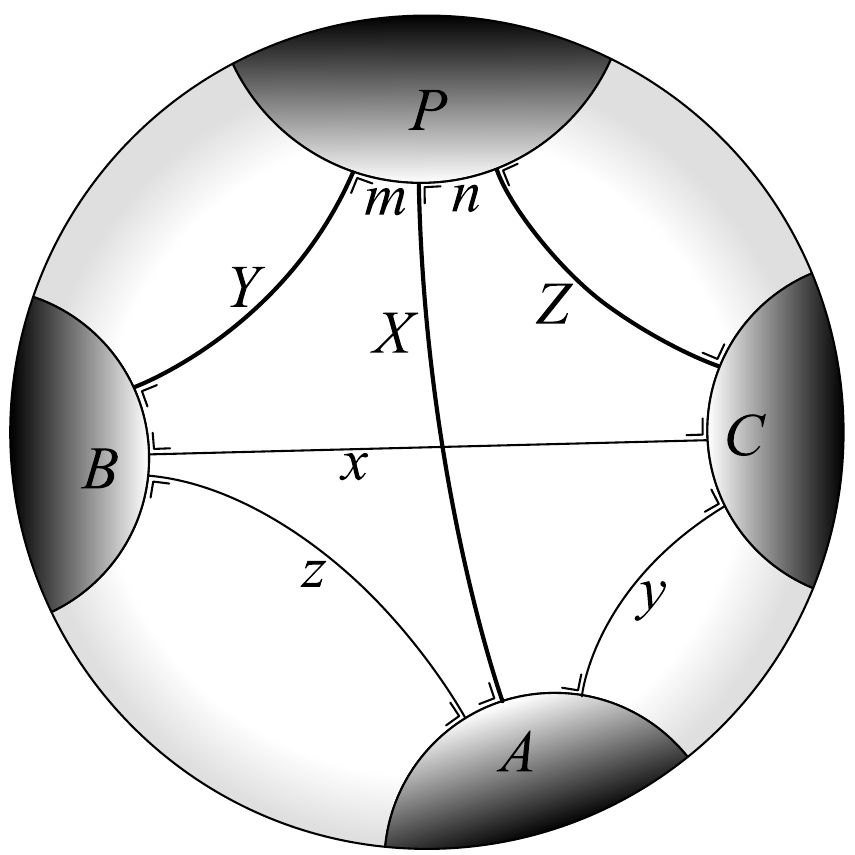}
        \captionsetup{width=1\linewidth} % Set the width of the caption
        \caption{Ptolemy relation of geodesics.} % Adjusted caption
        \label{geo} % Label for referencing
    \end{minipage}
\end{figure}

\begin{lem}\label{harmonic} If $B,A,C,P$ are in a cyclic order as in Figure \ref{geo} then one has a harmonic relation as follows:
$$\arccosh\frac{x+YZ}{\sqrt{(Y^2-1)(Z^2-1)}}=\arccosh \frac{z+XY}{\sqrt{(X^2-1)(Y^2-1)}}+\arccosh \frac{y+XZ}{\sqrt{(X^2-1)(Z^2-1)}}.$$
\end{lem}
\begin{proof}
Let $m,n$ be two segments respectively between $X,Y$ and $X,Z$ on $P$. The harmonic relation follows from computing the length of $m$, $n$, and $m+n$ by using the hyperbolic trigonometric formula for right-angled hexagons.
\end{proof}

\begin{lem}\label{recursive}(Ptolemy relation of geodesics)
If $B,A,C,P$ are in a cyclic order as in Figure \ref{geo}, then
 $$X=\frac{(xy+z)Y+(xz+y)Z+\mathcal{P}(x,y,z)\mathcal{P}(x,Y,Z)}{x^2-1},$$
 where $\mathcal{P}(a,b,c):=\sqrt{a^2+b^2+c^2+2abc-1}$ for any $a,b,c$.
\end{lem}
\begin{proof}
Since $\cosh(m+n)=\cosh(m)\cosh(n)+\sinh(m)\sinh(n)$,
  $$\frac{x+YZ}{\sqrt{(Y^2-1)(Z^2-1)}}=\frac{z+XY}{\sqrt{(X^2-1)(Y^2-1)}}.\frac{y+XZ}{\sqrt{(X^2-1)(Z^2-1)}}$$
  \begin{align*}
 & +\sqrt{\left(\left(\frac{z+XY}{\sqrt{(X^2-1)(Y^2-1)}}\right)^2-1\right)\left(\left(\frac{y+XZ}{\sqrt{(X^2-1)(Z^2-1)}}\right)^2-1\right)}.
 \end{align*}
Equivalently, \begin{equation}\label{eq1}(x+YZ)(X^2-1)-(z+XY)(y+XZ)=\mathcal{P}(X,Y,z)\mathcal{P}(X,y,Z).\end{equation}
Under the condition $X,Y,Z,x,y,z>1$, Equation \ref{eq1} can be solved to obtain the formula of $X$ in terms of $x,y,z,Y,Z$ as desired.
\end{proof}
\begin{cor}\label{triangle}(Quadruplet of geodesics) The following equation holds for any order of $P,A,B,C$:
\begin{equation*}
    \label{triangleq}
(x^2-1)X^2+(y^2-1)Y^2+(z^2-1)Z^2-2(xy+z)XY-2(yz+x)YZ-2(xz+y)XZ
=x^2+y^2+z^2+2xyz-1.
\end{equation*}

This equation can be expressed in a matrix form: $$\textbf{x}^TM\textbf{x}=\sqrt{-\det(M)},$$
where 
$$M=M(x,y,z):=\begin{bmatrix}
 x^2-1& -xy-z & -xz-y\\
-yx-z & y^2-1 & -yz-x\\
-zx-y&-zy-x&z^2-1
\end{bmatrix} \text{ and } \textbf{x}:=\begin{bmatrix}X\\Y\\Z\end{bmatrix}.$$

%The equation is also equivalent to:
%$$\det\begin{bmatrix}
 %-1/2& 1 & 1&1&1\\
%1 & 0 & -2z-2 & -2y-2&-2X-2\\
%1&-2z-2&0&-2x-2&-2Y-2\\
%1&-2y-2&-2x-2&0&-2Z-2\\
%1&-2X-2&-2Y-2&-2Z-2&0
%\end{bmatrix}=0,$$
%or equivalently, 
%$$\det\begin{bmatrix}
 %2& 1 & 1&1&1\\
%1 & 0 & \overline{z}^2 & \overline{y}^2&\overline{X}^2\\
%1&\overline{z}^2&0&\overline{x}^2&\overline{Y}^2\\
%1&\overline{y}^2&\overline{x}^2&0&\overline{Z}^2\\
%1&\overline{X}^2&\overline{Y}^2&\overline{Z}^2&0
%\end{bmatrix}=0,$$
%or
 
%$\overline{x}^2(\overline{x}^2-1)\overline{X}^4+\overline{y}^2(\overline{y}^2-1)\overline{Y}^4+\overline{z}^2(\overline{z}^2-1)\overline{Z}^4-(2\overline{x}^2\overline{y}^2+\overline{z}^2-\overline{x}^2-\overline{y}^2)\overline{X}^2\overline{Y}^2-(2\overline{y}^2\overline{z}^2+\overline{x}^2-\overline{y}^2-\overline{z}^2)\overline{Y}^2\overline{Z}^2-(2\overline{x}^2\overline{z}^2+\overline{y}^2-\overline{x}^2-\overline{z}^2)\overline{X}^2\overline{Z}^2+\overline{x}^2(\overline{y}^2+\overline{z}^2-\overline{x}^2)\overline{X}^2+\overline{y}^2(\overline{x}^2+\overline{z}^2-\overline{y}^2)\overline{Y}^2+\overline{z}^2(\overline{x}^2+\overline{y}^2-\overline{z}^2)\overline{Z}^2=\overline{x}^2\overline{y}^2\overline{z}^2.$
\end{cor}
\begin{proof} If $B,A,C,P$ are in a cyclic order as in Figure \ref{geo}, one has Equation \ref{eq1}. Taking the square of both sides of Equation \ref{eq1} and simplifying the result, one obtains the equation as desired. Since this equation is symmetric in term of $X,Y,Z$ and $x,y,z$, it holds for any order of $A,B,C,P$. 
\end{proof}
\textbf{Remark.} 
\bi
\item This equation has two roots in the variable $X$. If $B,A,C,P$ are in a cyclic order, $$X=\frac{(xy+z)Y+(xz+y)Z+\mathcal{P}(x,y,z)\mathcal{P}(x,Y,Z)}{x^2-1}.$$
Otherwise,
 $$X=\frac{(xy+z)Y+(xz+y)Z-\mathcal{P}(x,y,z)\mathcal{P}(x,Y,Z)}{x^2-1}.$$
\item $(-1,z,y),(z,-1,x)$, and $(y,x,-1)$ are (fundamental) solutions of this equation.
\ei
\subsection{Mixed Ptolemy relations}\label{mixedPtolemy}
Let $P,A,B$ be three disjoint geodesics and horocycles in $\mathbb{H}$. Each of them divides $\mathbb{H}$ into two domains such that the other two lie in the same domain. We denote by $PA,PB,AB$ the three shortest geodesic arcs connecting $P$ and $A$, $P$ and $B$, $A$ and $B$, respectively. Let $m$ be the length of the segment between $PA,PB$ on $P$. With the notations in Items \ref{3.1a} and \ref{3.2a} of Section \ref{Notations}, using standard calculations in the upper half-plane model, one can compute $m$ in terms of $PA,PB$ and $AB$.
\begin{lem}\label{computem}
\begin{itemize}
\item If $P,A,B$ are horocycles, then from \cite{penner1987decorated}: $m=\sqrt{\dfrac{AB}{2.PA.PB}}.$
\item If $P,A$ are horocycles, and $B$ is a geodesic, then: $m=\sqrt{\dfrac{1}{4.PB^2}+\dfrac{AB}{2.PA.PB}}.$
\item If $P$ is a horocycle, and $A, B$ are geodesics, then: $m=\sqrt{\dfrac{1}{4.PA^2}+\dfrac{1}{4.PB^2}+\dfrac{AB}{2.PA.PB}}.$
\item If $P$ is a geodesic, and $A, B$ are horocycles, then: $m=\arccosh\left(1+\dfrac{AB}{PA.PB}\right).$
\item If $P,A$ are geodesics, and $B$ is a horocycle, then: $m=\arccosh\left(\dfrac{PA.PB+AB}{PB\sqrt{PA^2-1}}\right).$
%\item If $P,A,B$ are points, then: $\measuredangle (PA,PB)=\arccos\left(\dfrac{PA.PB-AB}{\sqrt{(PA^2-1)(PB^2-1)}}\right).$
\end{itemize}
\end{lem}
\textbf{Remark.} Using these formulae for $m$ and the argument in Lemma \ref{harmonic}, one can establish different forms of harmonic relation depending on whether $P,A,B,C$ are horocycles and/or geodesics. In total, there are 10 different forms of the harmonic relation including also the two cases where $P,A,B,C$ are all geodesics/horocycles discussed before.

Let $A,B,C,P$ be four disjoint geodesics and horocycles in $\mathbb{H}$. Each of them divides $\mathbb{H}$ into two domains such that the other three lie in the same domain. Assume that $P,B,A,C$ are in a cyclic order. By combining Lemma \ref{computem} and the argument in Lemma \ref{recursive}, one can express $X:=PA$ in terms of $Y:=PB,Z:=PC,x:=AB,y:=AC,z:=BC$ in several different cases
\begin{lem}\label{mix}(Mixed Ptolemy and quadruplet relations) 
\vspace{-3mm}
\begin{itemize}
\item If $P,A,C$ are geodesics, and $B$ is a horocycle, then
\begin{itemize}
\item $X=\cfrac{xyY+xzZ+zY+\sqrt{(x^2+z^2+2xyz)(x^2+Y^2+2xYZ)}}{x^2},$
\item $x^2X^2+(y^2-1)Y^2+z^2Z^2-2XY(xy+z)-2YZ(yz+x)-2xzXZ=x^2+z^2+2xyz.$
\end{itemize}
\item If $P,B,C$ are geodesics, and $A$ is a horocycle, then
\begin{itemize}
\item $X=\cfrac{xyY+xzZ+zY+yZ+\sqrt{(y^2+z^2+2xyz)(x^2+Y^2+Z^2+2xYZ-1)}}{x^2-1},$
\item $(x^2-1)X^2+y^2Y^2+z^2Z^2-2XY(xy+z)-2yzYZ-2(xz+y)XZ=y^2+z^2+2xyz.$
\end{itemize}
\item If $P,B$ are geodesics, and $A,C$ are horocycles, then
\begin{itemize}
\item $X=\cfrac{xyY+xzZ+yZ+\sqrt{(y^2+2xyz)(x^2+Z^2+2xYZ)}}{x^2},$ 
\item $x^2X^2+y^2Y^2+z^2Z^2-2xyXY-2yzYZ-2(xz+y)XZ=y^2+2xyz.$
\end{itemize}
\item If $P,A$ are geodesics, and $B,C$ are horocycles, then
\begin{itemize}
\item $X=\cfrac{yY+zZ+\sqrt{(x+2yz)(x+2YZ)}}{x},$
\item $x^2X^2+y^2Y^2+z^2Z^2-2xyXY-2(yz+x)YZ-2xzXZ=x^2+2xyz.$
\end{itemize}
\item If $P$ is a geodesic, and $A,B,C$ are horocycles, then
\begin{itemize}
\item $X=\cfrac{yY+zZ+\sqrt{2yz(x+2YZ)}}{x},$
\item $x^2X^2+y^2Y^2+z^2Z^2-2xyXY-2yzYZ-2xzXZ=2xyz.$
\end{itemize}
\item If $P$ is a horocycle, and $A,B,C$ are geodesics, then
\begin{itemize}
\item $X=\cfrac{xyY+xzZ+zY+yZ+\sqrt{(x^2+y^2+z^2+2xyz-1)(Y^2+Z^2+2xYZ)}}{x^2-1},$
\item $(x^2-1)X^2+(y^2-1)Y^2+(z^2-1)Z^2-2(xy+z)XY-2(yz+x)YZ-2(xz+y)XZ=0.$
\end{itemize}
\item If $P,A$ are horocycles, and $B,C$ are geodesics, then
\begin{itemize}
\item $X=\cfrac{xyY+xzZ+zY+yZ+\sqrt{(y^2+z^2+2xyz)(Y^2+Z^2+2xYZ)}}{x^2-1},$
\item $(x^2-1)X^2+y^2Y^2+z^2 Z^2-2(xy+z)XY-2yzYZ-2(xz+y)XZ=0.$
\end{itemize}
\item If $P,A,B$ are horocycles, and $C$ is a geodesic, then
\begin{itemize}
\item $X=\cfrac{xyY+xzZ+zY+\sqrt{(z^2+2xyz)(Y^2+2xYZ)}}{x^2},$ 
\item $(x^2-1)X^2+y^2Y^2+z^2 Z^2-2(xy+z)XY-2yzYZ-2xzXZ=0.$
\end{itemize}
\item If $P,A,B,C$ are horocycles, then from Equation \ref{PennerPto}, we have
\begin{itemize}
\item $X=\cfrac{(\sqrt{yY}+\sqrt{zZ})^2}{x},$ 
\item $x^2X^2+y^2Y^2+z^2 Z^2-2xyXY-2yzYZ-2xzXZ=0.$
\end{itemize}
%\item If $P,A,B,C$ are points, then\begin{flalign*} X=\frac{xyY+xzZ-zY-yZ+\sqrt{(x^2+y^2+z^2-2xyz-1)(x^2+Y^2+Z^2-2xYZ-1)}}{x^2-1}.&& 
%\end{flalign*} 
%$$(x^2-1)X^2+(y^2-1)Y^2+(z^2-1)Z^2-2(xy-z)XY-2(yz-x)YZ-2(xz-y)XZ=x^2+y^2+z^2-2xyz-1.$$
\end{itemize}
\end{lem}

%The unique formula will be \begin{lem}{\tiny$X^2=\frac{\sqrt{\text{kA}^2 x^4-2 \text{kA} x^2 \left(\text{kB}y^2+\text{kC} z^2\right)+\text{kB}^2 y^4-2 \text{kB} \text{kC}y^2 z^2+\text{kC}^2 z^4+4 x^2 y^2 z^2} \sqrt{\text{kB}^2 Z^4-2\text{kP} x^2 \left(\text{kB} Z^2+\text{kC} Y^2\right)-2\text{kB} \text{kC} Y^2 Z^2+\text{kC}^2 Y^4+\text{kP}^2 x^4+4x^2 Y^2 Z^2}+\text{kA} x^2 \left(\text{kB} Z^2+\text{kC}Y^2-\text{kP} x^2\right)-\text{kB}^2 y^2 Z^2+\text{kB} \text{kC}y^2 Y^2+\text{kB} \text{kC} z^2 Z^2+\text{kB} \text{kP} x^2y^2-\text{kC}^2 Y^2 z^2+\text{kC} \text{kP} x^2 z^2-2 x^2 y^2Y^2-2 x^2 z^2 Z^2}{2 \text{kB} \text{kC} x^2-2 x^4}$}\end{lem}
\textbf{Remark.} Since the curvatures of geodesic and horocycle are $0$ and $1$ respectively, one can try to get a unifying formula that is applied in all cases of curves of constant curvatures in $[-1,1]$. One may also try to generalize the notion of signed distance between a horocycle and a geodesic mentioned in \cite{springborn2018hyperbolic}. In section \ref{determinant} we will present a unifying formula for all relations in Lemma \ref{mix} in the form of the Cayley-Menger determinant.

\subsection{A quintet of geodesics}\label{subsection 5g}
Let $A,P,B,C,D$ be five disjoint geodesics in $\mathbb{H}$ with a cyclic order. Each of them divides $\mathbb{H}$ into two domains such that the other four geodesics lie in the same domain. Let $O$ be the intersection between $AC$ and $BD$. With the notations in Items \ref{3.1a}, \ref{3.1b} and \ref{3.1c} in Section \ref{Notations}, one has the following relations:
\begin{lem}\label{quadrilateral}(Orthoquadrilateral relation)
$$\frac{PA.OC+PC.OA}{PB.OD+PD.OB}=\frac{\widetilde{AC}}{\widetilde{BD}}.$$
\end{lem}
\begin{proof}
Let $\alpha$ and $\beta$ be the angles between $OA,OP$ and $OB,OP$ as in Figure \ref{fig:quadri}.

\begin{figure}[H]
\begin{minipage}[t]{0.5\textwidth}
  \centering
  % First figure
  %\ShowGrid
  \leavevmode \SetLabels
  \L(.4*.34) $\alpha$\\%
  \L(.47*.36) $\beta$\\
  \endSetLabels
  \AffixLabels{\includegraphics[width=4cm,angle=0]{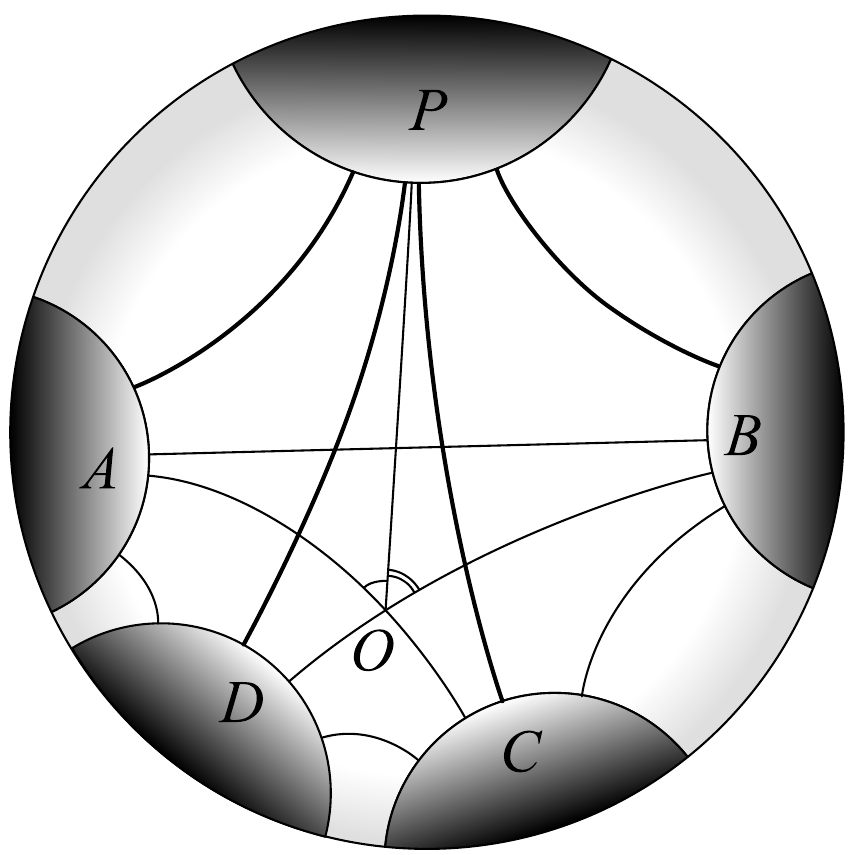}}
  % \vspace{-24pt}
  \caption{A quintet of geodesics.}
  \label{fig:quadri}
\end{minipage}%
\begin{minipage}[t]{0.5\textwidth}
  \centering
  % Second figure
  %\ShowGrid
  \leavevmode \SetLabels
  \L(.505*.4) $\alpha$\\%
  \L(.57*.415) $\beta$\\
  \endSetLabels
  \AffixLabels{\includegraphics[width=4cm,angle=0]{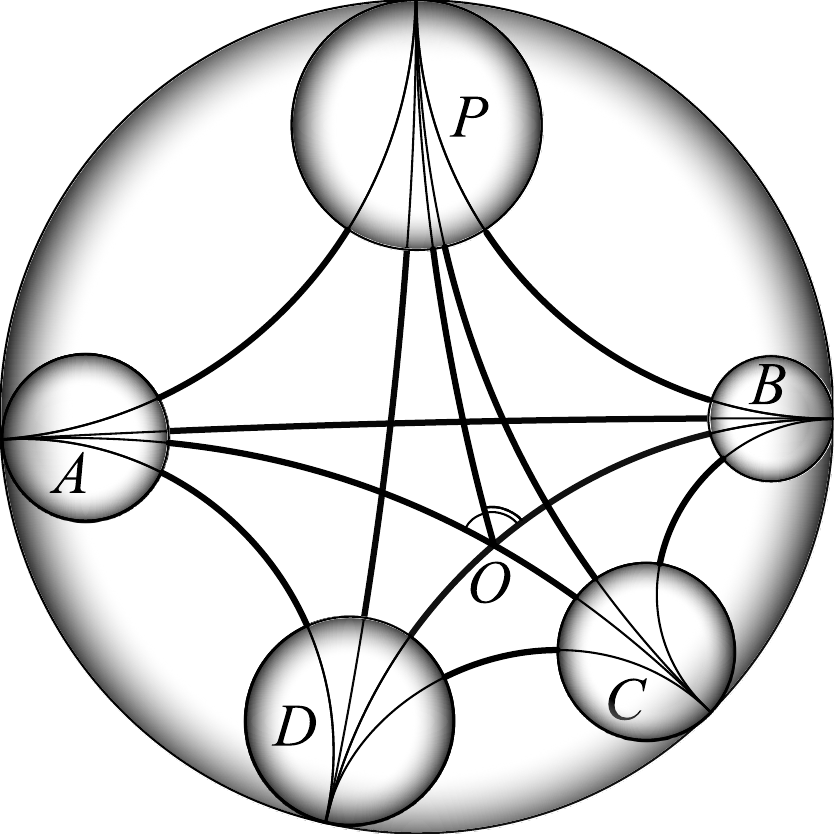}}
  % \vspace{-24pt}
  \caption{A quintet of horocycles.}
  \label{fig:horo}
\end{minipage}
\end{figure}

Using the trigonometry formulae for hyperbolic pentagons with four right-angles (see \cite{beardon2012geometry,buser2010geometry}), one obtains:
$$PA=-OP.OA.\cos(\alpha)+\widetilde{OP}.\widetilde{OA};\,\,\,\,\,\,\,\, PB=-OP.OB.\cos(\beta)+\widetilde{OP}.\widetilde{OB};$$
$$PC=-OP.OC.\cos(\pi-\alpha)+\widetilde{OP}.\widetilde{OC};\,\,\,\,\,\,\,\, PD=-OP.OD.\cos(\pi-\beta)+\widetilde{OP}.\widetilde{OD}.$$
These equations imply that:
$$PA.OC+PC.OA=\widetilde{OP}(\widetilde{OA}.OC+\widetilde{OC}.OA),$$
$$PB.OD+PD.OB=\widetilde{OP}(\widetilde{OD}.OB+\widetilde{OB}.OD).$$
And so
$$\frac{PA.OC+PC.OA}{PB.OD+PD.OB}=\frac{\widetilde{OA}.OC+\widetilde{OC}.OA}{\widetilde{OD}.OB+\widetilde{OB}.OD}=\frac{\widetilde{AC}}{\widetilde{BD}}.$$
\end{proof}
The following are special cases of the orthoquadrilateral relation.
\begin{cor}\label{all}%(  and orthokite relation of geodesics) 
With $P,A,B,C,D$ defined above, one has
\vspace{-3mm}
\begin{itemize}
\item[1.](Ortho-isosceles trapezoid): If $AD=BC$ and $AC=BD$ then: $$\dfrac{\overline{PA}^2-\overline{PB}^2}{\overline{AB}}=\dfrac{\overline{PD}^2-\overline{PC}^2}{\overline{CD}}.$$
\item[2.](Orthorectangle): If $AB=CD$, $AD=BC$ and $AC=BD$ then: $$\overline{PA}^2+\overline{PC}^2=\overline{PB}^2+\overline{PD}^2.$$
\item[3.](Orthoparallelogram): If $AB=CD$, $AD=BC$ then: $$\dfrac{PA+PC}{PB+PD}=\sqrt{\dfrac{AC-1}{BD-1}}.$$
\item[4.](Orthokite): If $AB=AD$, $CB=CD$ then: $$PD+PB=\dfrac{2(AC.BC+AB)}{AC^2-1}.PA+\dfrac{2(AC.AB+BC)}{AC^2-1}.PC.$$
\end{itemize}
\end{cor}
\begin{proof}
1. By using hyperbolic trigonometry for pentagons and right-angled hexagons, from the conditions $AD=BC$ and $AC=BD$, one can show that $OA=OB$, $OC=OD$ and then $$\frac{PA.OC+PC.OA}{PB.OC+PD.OA}=1.$$The last equality implies that:\,\,\,\, $\dfrac{OA}{OC}=\dfrac{PA-PB}{PD-PC}=\dfrac{\overline{PA}^2-\overline{PB}^2}{\overline{PD}^2-\overline{PC}^2}.$

We also have
$$AB=-OA.OB.\cos(\alpha+\beta)+\widetilde{OA}.\widetilde{OB}=-OA^2.\cos(\alpha+\beta)+OA^2-1,$$
$$CD=-OC.OD.\cos(\alpha+\beta)+\widetilde{OC}.\widetilde{OD}=-OC^2.\cos(\alpha+\beta)+OC^2-1.$$
And so: \,\,\,\,\,$\dfrac{\overline{AB}}{\overline{CD}}=\dfrac{OA}{OC}.$

3. By using hyperbolic trigonometry in pentagons and right-angled hexagons, from the conditions $AB=CD$ and $AD=BC$, one can show that $OA=OC=\overline{AC}$, $OC=OD=\overline{CD}$ and hence $$\frac{PA+PC}{PB+PD}=\frac{\overline{BD}}{\overline{AC}}.\frac{\widetilde{AC}}{\widetilde{BD}}=\sqrt{\frac{\overline{AC}^2-1}{\overline{BD}^2-1}}.$$
4. Denote $b:=AB=AD$, $c:=CB=CD$ and $a:=AC$. One can show that $PD>PB$, then by applying the relation of a quadruplet of geodesics (Corollary \ref{triangle}) for two quadruples $\{P;A,C,D\}$ and $\{P;A,C,B\}$, one has that $PB$ and $PD$ are two different roots of the following quadratic equation of variable $t$:
\begin{align*}
&PA^2(c^2-1)+PC^2(b^2-1)+t^2(a^2-1)\\
&=2(bc+a).PA.PC+2(ac+b).PA.t+2(ab+c).PC.t+a^2+b^2+c^2+2abc-1.
\end{align*}
Thus the sum of the two roots gives us the relation: $$PD+PB=\dfrac{2(ac+b)}{a^2-1}.PA+\dfrac{2(ab+c)}{a^2-1}.PC.$$
\end{proof}
\textbf{Remark.} If $P$ is identical to one of the four geodesics, then the orthorectangle relation becomes the Pythagorean relation. %Also note that for an ortho-isosceles-trapezoid ($AD=BC$, $AC=BD$), if there exists a geodesic $P$ such that $PA=PD$ and $PB=PC$, then by the relation (1) in Corollary \ref{all}, $AB=CD$. One can apply this observation to show the equality of the lengths of some orthogeodesics on a surface. However, this way is just a special case of using reflection involutions as we will see in Section \ref{reflection}.
\subsection{A quintet of horocycles}\label{subsection 5h}
Let $A,P,B,C,D$ be five disjoint horocycles in $\mathbb{H}$ with a cyclic order. Each of them divides $\mathbb{H}$ into two domains such that the other four horocycles lie in the same domain. Let $O$ be the intersection between $AC$ and $BD$. Let $\alpha$ and $\beta$ be the angles between $OA,OP$ and $OB,OP$ respectively as in Figure \ref{fig:horo}.

With the notations in Section \ref{Notations} (Items \ref{3.2a} and \ref{3.2c}), and using standard calculations in the upper half-plane model of the hyperbolic plane, the following relations are obtained:
\begin{lem}\label{horopentagon} With $P,A,O$ and $\alpha$ defined above, one has
$$PA=OP.OA(1-\cos \alpha).$$
\end{lem}
Applying Lemma \ref{horopentagon} and the Penner's Ptolemy relation, one can compute $OA,OB,OC,OD$ in terms of $AC,BD,AD,AB,CD,CB$ as follows:
\begin{lem}\label{OA} With $P,A,B,C,D$ and $O$ defined above, one obtains
$$OA=\frac{\overline{AC}}{2}\sqrt{\frac{\overline{AD}.\overline{AB}}{\overline{CD}.\overline{CB}}};\,\,\, OC=\frac{\overline{AC}}{2}\sqrt{\frac{\overline{CD}.\overline{CB}}{\overline{AD}.\overline{AB}}};\,\,\, OB=\frac{\overline{BD}}{2}\sqrt{\frac{\overline{BA}.\overline{BC}}{\overline{DA}.\overline{DC}}};\,\,\, OD=\frac{\overline{BD}}{2}\sqrt{\frac{\overline{DA}.\overline{DC}}{\overline{BA}.\overline{BC}}}.$$
\end{lem}
Applying these above formulae and the argument as in the proof of Lemma \ref{quadrilateral}, one obtains the following relation:
\begin{lem}\label{quadrilateral*}(Orthoquadrilateral relation of horocycles) With $P,A,B,C,D$ and $O$ defined above, one has
$$\frac{PA.OC+PC.OA}{PB.OD+PD.OB}=\frac{AC}{BD}.$$
Equivalently, $$\frac{PA.\overline{CD}.\overline{CB}+PC.\overline{AD}.\overline{AB}}{PB.\overline{DA}.\overline{DC}+PD.\overline{BA}.\overline{BC}}=\frac{\overline{AC}}{\overline{BD}}.$$
\end{lem}
Similarly, one also has the following properties of Penner's lambda lengths in special cases:
\begin{cor}\label{isosceles*}
With $P,A,B,C,D$ defined above, one has
\vspace{-3mm}\begin{itemize}
\item{}(Ortho-isosceles trapezoid): If $AD=BC$ and $AC=BD$ then: $$\dfrac{\overline{PA}^2-\overline{PB}^2}{\overline{AB}}=\dfrac{\overline{PD}^2-\overline{PC}^2}{\overline{CD}}.$$
\item{}(Orthorectangle): If $AB=CD$, $AD=BC$ and $AC=BD$ then: $$\overline{PA}^2+\overline{PC}^2=\overline{PB}^2+\overline{PD}^2.$$
\item{}(Orthoparallelogram): If $AB=CD$, $AD=BC$ then: $$\dfrac{\overline{PA}^2+\overline{PC}^2}{\overline{PB}^2+\overline{PD}^2}=\dfrac{\overline{AC}}{\overline{BD}}.$$
\item (Orthokite): If $AB=AD$, $CB=CD$ then: $$PD+PB=\sqrt{\dfrac{BD.BC}{AC.AB}}.PA+\sqrt{\dfrac{BD.AB}{AC.BC}}.PC.$$
\end{itemize}
\end{cor}

\subsection{A unifying formula in the form of the Cayley-Menger determinant}\label{determinant}
In this section, one will see that all of the relations in previous sections can be put into a unifying formula in the form of the Cayley-Menger determinant. Let $|UV|$ be the Euclidean distance between two arbitrary points $U$ and $V$ in the Euclidean space. In the field of distance geometry, the Cayley-Menger determinant allows us to compute the volume of an $n-$simplex in the Euclidean space in terms of the squares of all the distances between pairs of its vertices. In a special case, if $A,B,C,D$ are four points in the Euclidean plane, one has a relation as follows:
$$\det\begin{bmatrix}
 0& 1 & 1&1&1\\
1 & 0 & |AB|^2 & |AC|^2&|AD|^2\\
1&|AB|^2&0&|BC|^2&|BD|^2\\
1&|AC|^2&|BC|^2&0&|CD|^2\\
1&|AD|^2&|BD|^2&|CD|^2&0
\end{bmatrix}=0.$$ 
If $A,B,C,D$ are four points in the hyperbolic plane, with the notation in Item \ref{3.3a} of Section \ref{Notations}, one has a similar formula  (see \cite{blumenthal1943distribution}) as follows:
$$\det\begin{bmatrix}
-2& 1 & 1&1&1\\
1 & 0 & \overline{AB}^2 & \overline{AC}^2&\overline{AD}^2\\
1&\overline{AB}^2&0&\overline{BC}^2&\overline{BD}^2\\
1&\overline{AC}^2&\overline{BC}^2&0&\overline{CD}^2\\
1&\overline{AD}^2&\overline{BD}^2&\overline{CD}^2&0
\end{bmatrix}=0.$$
Now we consider geodesics and horocycles in the hyperbolic plane. Let $X$ be a curve of constant curvature in $\mathbb{H}$, denote by $\kappa_{X}$ the geodesic curvature of $X$. Thus $\kappa_{X}=0$ or $1$ if $X$ is a geodesic or a horocycle respectively. With the notations in Items \ref{3.1b} and \ref{3.2c} of Section \ref{Notations}, the relations of quadruplet of geodesics/horocycles in Lemma \ref{mix} can be written in a unifying form as follows
\begin{thm}\label{cayle}
 Let $A,B,C,D$ be four disjoint geodesics/horocycles, each of them divides $\mathbb{H}$ into two domains such that the other three lie in the same domain. Then
$$\det\begin{bmatrix}
2& 1-\kappa_{A} & 1-\kappa_{B}&1-\kappa_{C}&1-\kappa_{D}\\
1-\kappa_{A} & 0 & \overline{AB}^2 & \overline{AC}^2&\overline{AD}^2\\
1-\kappa_{B}&\overline{AB}^2&0&\overline{BC}^2&\overline{BD}^2\\
1-\kappa_{C}&\overline{AC}^2&\overline{BC}^2&0&\overline{CD}^2\\
1-\kappa_{D}&\overline{AD}^2&\overline{BD}^2&\overline{CD}^2&0
\end{bmatrix}=0.$$
\end{thm}
\begin{proof}
Direct verification of all the cases (Lemma \ref{mix}).
\end{proof}
\textbf{Remark.} 
We suspect that the result can be generalized to higher dimensions in the following form: 
Let $U$ and $V$ be two arbitrarily disjoint hyperbolic spaces or horoballs of co-dimension 1 in $\mathbb{H}^n$. Denote \(\kappa_U\) and \(\kappa_V\) as their respect geodesic curvatures. We define $\overline{UV}$ as follows: $$\overline{UV} :=\frac{e^{\frac{1}{2}d_{\mathbb{H}}(U,V)}+(1-\kappa_U)(1-\kappa_V)e^{-\frac{1}{2}d_{\mathbb{H}}(U,V)}}{2},$$ where $d_{\mathbb{H}}(U,V)$ is the hyperbolic distance between $U$ and $V$. Let $\{A_1,A_2,...,A_k,A_{k+1},A_{k+2}\}$ be the set of $k+2$ disjoint hypersurfaces of constant geodesic curvatures in $n$-dimensional hyperbolic space $\mathbb{H}^n$ ($k \geq n$), each of them divides $\mathbb{H}^n$ into two domains such that the other $k+1$ hypersurfaces lie in the same domain. If $\kappa_{A_i}\in\{0,1\}$ for all $i\in \{1,2,...,k+2\}$, then  

$$\det\begin{bmatrix}
2& 1-\kappa_{A_1} & 1-\kappa_{A_2}&1-\kappa_{A_3}&\cdots &1-\kappa_{A_{k+2}}\\
1-\kappa_{A_1} & 0 & \overline{A_1A_2}^2 & \overline{A_1A_3}^2& \cdots &\overline{A_1A_{k+2}}^2\\
1-\kappa_{A_2}&\overline{A_2A_1}^2&0&\overline{A_2A_3}^2& \cdots &\overline{A_2A_{k+2}}^2\\
1-\kappa_{A_3}&\overline{A_3A_1}^2&\overline{A_3A_2}^2&0&\cdots &\overline{A_3A_{k+2}}^2\\
  \vdots&  \vdots&  \vdots& \vdots& \ddots & \vdots \\
1-\kappa_{A_{k+2}}&\overline{A_{k+2}A_1}^2&\overline{A_{k+2}A_2}^2 & \overline{A_{k+2}A_{3}}^2& \cdots&0
\end{bmatrix}=0.$$
Note that, this formula holds for the case that $n =2, k=3$ and $A_i$'s are all either geodesics or horocycles as proved in Subsections \ref{subsection 5g} and \ref{subsection 5h}.

\section{Further questions and remarks} \label{questionremarks}
%\subsection{Questions following from Section \ref{section 1}}
%Given an orthobasis on a hyperbolic surface, find an algorithm to obtain a nicest orthobasis.
\subsection{Questions following from Section \ref{reflection}}
In a similar vein to Theorem \ref{thm:reflection} in Section \ref{reflection}, one can ask the following question
\begin{que}
Is it true that: 
\bi
\item[a.] All e-ortho-isosceles-trapezoids, e-orthokites, and e-orthorectangles arise from reflection involutions; 
\item[b.] There is no e-orthorhombus; 
\item[c.] If $S$ is not a one-holed torus, then any e-orthoparallelogram on $S$ is an e-orthorectangle. 
\ei
\end{que}

One can also come up with a conjecture closely related to the length equivalent problem of closed geodesics on hyperbolic surfaces stated in \cite{anderson2003variations} and \cite{leininger2003equivalent}:
\begin{conj}
Two orthogeodesics $\alpha$ and $\beta$ are length equivalent if and only if there is a finite sequence of reflection involutions $r_1,r_2,...,r_n$ so that $r_1 \circ r_2 \circ ... \circ r_n (\alpha)=\beta$. In other words, there is a finite sequence of e-ortho-isosceles-trapezoids ``connecting'' $\alpha$ and $\beta$. 

% \comment{How to translate this to the group theoretical point of view? Infinitely many generators? Coxeter group? Thin group?}
\end{conj} 
Note that each e-ortho-isosceles-trapezoid consists of 6 orthogeodesics which are associated with at most 8 closed geodesics (including also at most two simple closed geodesics at the boundary). Thus the identity relations of orthogeodesics can be translated to trace-type identities of closed geodesics. Also, each closed geodesic is a boundary of infinitely many immersed pairs of pants, and each of them is associated with an (improper) orthogeodesics (an arc perpendicular to two closed geodesics), thus it would be interesting to study the set of all (improper) orthogeodesics.
\subsection{Questions following from Section \ref{integralspectrum}}
% If a surface is ortho-integral, then the Diophantine equation at each vertex of the orthotree is solvable in integers $(X,Y,Z)$. For example, in the case of a pair of pants with a standard orthobasis $a=b=c=3$, the vertex relation (see Section \ref{recursiveformula}) is:
% $$X^2+Y^2+Z^2-3XY-3YZ-3XZ=10.$$

Each orthobasis of an ortho-integral surface generates a collection of Diophantine equations.
\bq
 Is it true that we can find all the integer solutions (up to signs) of these Diophantine equations within the set of triples of traces of neighboring orthogeodesics?

\eq

Our current understanding of ortho-integral surfaces is limited to a finite list as presented in Theorem \ref{Thm: integral surfaces}. This limitation naturally leads to the following questions.
\begin{que}
Are there infinitely many ortho-integral surfaces? Is it true that surfaces formed by gluing any finite number of ortho-integral pairs of pants $P(2,2,2)$ (or $P(3,3,3)$) without twisting remain ortho-integral?
\end{que}

\subsection{Final remarks}
Many ideas in this paper are still in the early stages and require further generalization. As mentioned in the introduction, numerous intriguing number-theoretical questions arise in relation to the set of orthogeodesics on ortho-integral surfaces, akin to those found in integral Apollonian circle packing. Besides that, classifying all ortho-integral surfaces remains a significant challenge. We are excited to explore these questions further, and this will be the focus of our forthcoming work.

\bibliographystyle{amsplain}
{
\small
\bibliography{main}
}
 \footnotesize{\textsc{Department of Mathematics, National University of Singapore, Singapore}\\
\footnotesize{\textsc{\& Institute of Mathematics, Vietnam Academy of Science and Technology, Hanoi, Vietnam}\\
\textsc{Email address: dnminh@math.ac.vn}}

\end{document}